\documentclass[reqno]{amsart}

\usepackage[normalem]{ulem}%pour barrer du texte
\usepackage{diagbox}
\usepackage{amsmath,amssymb,color}
\usepackage{amsfonts, amscd, epsfig,enumerate,bm}
\usepackage{graphicx}
\usepackage{graphics}
\usepackage{color}
\usepackage{mathrsfs}
\usepackage{todonotes}
\usepackage{soul}
\usepackage[displaymath, mathlines,pagewise]{lineno}
\usepackage[usenames,dvipsnames]{pstricks}
\usepackage{pstricks-add}
\usepackage{epsfig}
\usepackage{pst-grad} % For gradients
\usepackage{pst-plot} % For axes
\usepackage[space]{grffile} % For spaces in paths
\usepackage{etoolbox} % For spaces in paths
\usepackage[margin=3cm]{geometry}
\makeatletter % For spaces in paths
\patchcmd\Gread@eps{\@inputcheck#1 }{\@inputcheck"#1"\relax}{}{}
\makeatother

\newtheorem{theorem}{Theorem}[section]

\newtheorem{remark}{Remark}[section]

\newtheorem{lemma}[theorem]{Lemma}
\newtheorem{proposition}[theorem]{Proposition}

\usepackage{tikz}
\usetikzlibrary{backgrounds}
\usetikzlibrary{patterns,fadings}
\usetikzlibrary{arrows,decorations.pathmorphing}
\usetikzlibrary{decorations}
\usetikzlibrary{calc}
\usetikzlibrary{shapes.misc}

\usepackage{setspace}
\definecolor{light-gray}{gray}{0.95}
\usepackage{float}
\usepackage[colorlinks=true,linkcolor=blue,citecolor=magenta]{hyperref}
\def\centerarc[#1](#2)(#3:#4:#5){\draw[#1] ($(#2)+({#5*cos(#3)},{#5*sin(#3)})$) arc (#3:#4:#5);}

\newcommand{\vertiii}[1]{{\left\vert\kern-0.25ex\left\vert\kern-0.25ex\left\vert #1 
		\right\vert\kern-0.25ex\right\vert\kern-0.25ex\right\vert}}

\numberwithin{equation}{section}
\numberwithin{figure}{section}

\newcommand{\bb}[1]{{\mathbb #1}}

\newcommand{\<}{\big\langle}
\renewcommand{\>}{\big\rangle}

\renewcommand{\epsilon}{\varepsilon}

\newcommand{\R}{\mathbb R}
\newcommand{\Z}{\mathbb Z}

\renewcommand{\P}{\mathbb P}

\newcommand{\E}{\mathbb E}

\newcommand{\gen}{\mathscr{L}}

\newcommand{\Acal}{\mathcal{A}}

\newcommand{\Ccal}{\mathcal{C}}
\newcommand{\Dcal}{\mathcal{D}}
\newcommand{\Fcal}{\mathcal{F}}
\newcommand{\Ical}{\mathcal{I}}

\newcommand{\Ocal}{\mathcal{O}}
\newcommand{\Qcal}{\mathcal{Q}}
\newcommand{\Scal}{\mathcal{S}}

\newcommand{\alphabf}{\bm{\alpha}}
\newcommand{\betabf}{\bm{\beta}}

\newcommand{\gammabf}{\bm{\gamma}}

\allowdisplaybreaks %%Pour \UTF{00E9}viter les grands espacements verticaux.

\title[Sample path MDP for the current and the tagged particle]{Sample path MDP for the current and the tagged particle in the SSEP}

\author{Xiaofeng Xue}
\address{School of Mathematics and Statistics, Beijing Jiaotong University, Beijing 100044, China.}
\email{xfxue@bjtu.edu.cn}

\author{Linjie Zhao}
\address{School of Mathematics and Statistics \& Hubei Key Laboratory of Engineering Modeling and Scientific Computing, Huazhong University of Science and Technology, Wuhan 430074, China}
\email{linjie\_zhao@hust.edu.cn}

\thanks{Acknowledgments. The project is partially  supported by the National Natural Science Foundation of China with Grant Numbers 11501542, 11971038, and 12371142. The authors also thank the financial support from the Fundamental Research Funds for the Central Universities in China and  Hubei Key Laboratory of Engineering Modeling and Scientific Computing.
}


\begin{document}
	
\maketitle

\begin{abstract}
We prove sample path moderate deviation principles (MDP) for the current and the tagged particle in the symmetric simple exclusion process, which extends the results in \cite{xue2023moderate}, where the 
MDP was only proved at any fixed time.
\end{abstract}

\section{Introduction}

The exclusion process is informally defined as a family of indistinguishable particles performing random walks on some graph subjected to the exclusion rule, \emph{i.e.}, there is at most one particle at each site.  It has been a long-standing problem to investigate the behaviors of a typical particle in the exclusion process, which is called the tagged particle. Due to the interactions with the other particles, the replacement of the tagged particle itself is not a Markov process. Despite that, much progress has been made since Spitzer introduced the exclusion process in \cite{spitzer70}. We refer the readers to \cite{liggett99,komorowski2012fluctuations} and references therein for an excellent survey on the tagged particle process.

The model is usually called the symmetric simple exclusion process (SSEP) if the underlying graph is the one-dimensional infinite lattice $\Z$ and a particle jumps to its left and right neighbors at equal rates $1/2$. Since the total number of particles is conserved by the dynamics, the Bernoulli product measure with constant density $\rho \in [0,1]$, denoted by $\nu_{\rho}$,  is invariant for the SSEP, see \cite{liggettips} for example. Assume the initial distribution of the SSEP is the Bernoulli product measure $\nu_{\rho}$ conditioned on having a particle at the origin, and let $X(t)$ be the position of this tagged particle at time $t$. Since particles cannot jump over each other in the SSEP, the tagged particle turns out to be sub-diffusive. Precisely speaking, Arratia \cite{arratia1983motion} and De Masi and Ferrari \cite{DeMasiFerrari02} showed that 
\[\frac{X_t}{t^{1/4}} \Rightarrow \mathcal{N} (0,\sigma_X^2), \quad t \rightarrow +\infty,\]
where $\mathcal{N} (0,\sigma_X^2)$ is the normal distribution with mean zero and variance  $\sigma_X^2 := \sqrt{2/\pi} (1-\rho)/\rho$.  The above central limit theorem was then extended to the invariance principle by Peligrad and Sethuraman \cite{peligrad2008fractional},
\[\Big\{ \frac{X_{tN^2}}{N^{1/2}}: 0 \leq t \leq T \Big\} \Rightarrow \big\{ B_t^{1/4}: 0 \leq t \leq T\big\}, \quad N \rightarrow +\infty,\]
where $B_\cdot^{1/4}$ is the fractional Brownian motion with Hurst index $1/4$ and $T > 0$ is fixed.  Jara and Landim \cite{jara2006nonequilibrium} also proved central limit theorems for the tagged particle when the initial distribution of the SSEP is non-equilibrium. Recently, Conroy and Sethuraman \cite{conroy2023gumbel} showed that the scaled position of the tagged particle converges to a Gumble limit law when the process starts from the step configuration. 

Large and moderate deviation behaviors of the tagged particle in the SSEP have also been studied.  The large deviation behaviors were investigated  in \cite{sethuraman2013large,imamura2017large}.  When the initial distribution of the SSEP is the Bernoulli product measure $\nu_{\rho}$ conditioned on having a particle at the origin, the same authors \cite{xue2023moderate} proved that for any $t > 0$, the sequence $\{X (tN^2)/a_N\}_{N \geq 1}$ satisfies moderate deviation principles with decay rate $a_N^2/N$ and with rate function $\sqrt{2\pi} \rho \alpha^2 / [4(1-\rho) \sqrt{t}]$, where $\sqrt{N} \ll a_N \ll N$. 

We comment that the above deviation results only considered the behaviors of the tagged particle at any fixed time $t$.  The main aim of this paper is to  investigate  sample path moderate deviations  for the tagged particle.  Roughly speaking, we proved that when the SSEP starts from the same initial distribution as in \cite{xue2023moderate},  the sequence of processes $\{X (tN^2)/a_N: 0 \leq t \leq T\}_{N \geq 1}$ satisfies moderate deviation principles with decay rate $a_N^2/N$, where $\sqrt{N \log N} \ll a_N \ll N$ and $T > 0$ is fixed. Moreover, the moderate deviation rate function is proportional to the large deviation rate function of the fractional Brownian motion with Hurst index $1/4$. See Theorem \ref{theorem MDP of tagged particle} for details.

Since the relative ordering of the particles is conserved in the SSEP, closely related to the position of the tagged particle is the current of the process.  This current-tagged particle relationship has been used in the above literature \cite{DeMasiFerrari02,jara2006nonequilibrium,sethuraman2013large,xue2023moderate}.  To prove Theorem \ref{theorem MDP of tagged particle}, we  also proved sample path moderate deviation principles for the current, which is the content of Theorem \ref{theorem MDP of current} and has its own interest.

The main idea of the proof is as follows. As in \cite{xue2023moderate}, by using moderate deviation principles from hydrodynamic limits of the SSEP \cite{gao2003moderate}, we first prove finite-dimensional moderate deviation principles for the current and the tagged particle, where the rate function is given by a variational formula. In \cite{xue2023moderate}, this variational problem was solved by constructing a minimizer of the problem directly and cannot be extended to the finite-dimensional case straightforwardly. Instead, we  use the Fourier approach and then obtain  an explicit formula for the finite-dimensional deviation rate function. Then, we need to prove the exponential tightness of the sample path of the current and the tagged particle, which requires more delicate analysis compared with \cite{xue2023moderate} since all the estimates should be uniform in time. Finally, using standard results from large deviations theory, we obtain another variational formula for the rate function of the sample path, which could be solved easily from known results on fractional Brownian motion.

Very recently, the first named  author in \cite{xue2023nonequilibrium} proved non-equilibrium moderate deviations from hydrodynamic limits of the SSEP.  It remains a challenge to extend the results in this paper to the non-equilibrium case by using the above results. Another possible direction is to consider the sample path moderate deviation principles for the tagged particle in the SSEP in random environments or in the asymmetric exclusion process. We leave those as future works. 

\subsection{Related literature} For the tagged particle in exclusion processes, Saada \cite{saada1987tagged} and Rezakhanlou \cite{rezakhanlou1994evolution} proved the law of large numbers.   When the process starts from its equilibrium measure,  except the literature mentioned above, central limit theorems and invariance principles were proved  by Kipnis and Varadhan \cite{kipnis1986central} in the symmetric case in all dimensions except the one-dimensional nearest neighbor case, by Varadhan \cite{varadhan1995self} in the asymmetric mean-zero case, by Kipnis  \cite{Kipnis86} in the one-dimensional nearest neighbor asymmetric case  and by Sethuraman, Varadhan and Yau \cite{sethuraman2000diffusive} in the asymmetric case in three or higher dimensions.  In dimensions $d \leq 2$ when the underlying random walk has a drift except for the one-dimensional nearest-neighbor case,  a full CLT or invariance principle remains open. However, Sethuraman \cite{sethuraman2006diffusive} proved that  the tagged particle is diffusive in this case.

Jara \cite{jara2008longjumps} and the second named author \cite{zhao2023motion} proved CLT for the tagged particle when particles perform random walks with long jumps. The tagged particle in the SSEP with bond disorder was considered in \cite{jara2008quenched}.  In \cite{landim1998driven,loulakis2002einstein,loulakis2005mobility}, the authors considered the tagged particle process in order to check the validity of the Einstein relation. The heat kernel bound for the tagged particle in the SSEP was also obtained in \cite{giunti2019heat}.  In \cite{seppalainen98c,sethuraman2023atypical}, the authors investigated large deviations of the tagged particle in the asymmetric exclusion process.  Besides on the integer lattice $\Z^d$, the tagged particle was also considered on  regular trees \cite{chen2019limit} and on Galton--Watson trees \cite{gantert2020speed}.

The behavior of the particle has also been investigated in other interacting particle processes, such as in zero range processes \cite{jara2009nonequilibrium,jara2013nonequilibrium,sethuraman2007diffusivity} or in stirring-exclusion processes \cite{chen2013limit}.

\subsection{Notation}   For integers $m,n > 0$ and $U,V \subset \R$, let $\Ccal^{m,n} (U \times V)$ be the space of functions on $U \times V$ which are continuously $m$ (resp. $n$) times differentiable  on the first (resp. second) variable, and let $\Ccal_c^{m,n} (U \times V)$ be those functions in $\Ccal^{m,n} (U \times V)$ with compact support. Let $\mathcal{S} (\bb{R})$ be the space of Schwartz functions on $\R$ and let $\mathcal{S}^\prime(\bb{R})$ be its dual space, \emph{i.e.}, the space of tempered distributions. For a metric space $(\Sigma,d(\cdot,\cdot))$, let $\mathcal{D} ([0,\infty),\Sigma)$ be the space of càdlàg functions equipped with the Skorohod topology. 

For any column vector $\alphabf  = (\alpha_1,\alpha_2,\ldots,\alpha_n)^T \in \mathbb{R}^n$, denote by  $\|\alpha\|_\infty :=  \sup_{1 \leq j \leq n} |\alpha_j|$ the uniform norm of $\alphabf$. For $a \in \R$, let $[a]$ be the largest integer smaller than or equal to $a$. For two positive sequences $\{a_N\}$ and $b_N$, we write $a_N \ll b_N$ if $\lim_{N \rightarrow \infty} a_N / b_N = 0$.

\subsection{Outline of the paper}  In Section \ref{sec:result}, we introduce the model and state the main results of the paper.  We recall moderate deviations from hydrodynamic limits of the SSEP and prove several basic properties of the rate function from hydrodynamic limits in Section \ref{sec:estimates}.  Section \ref{sec:variation} is devoted to identifying the rate function of the tagged particle or the current with a variational formula.  We prove the exponential tightness of the tagged particle and the current in Section \ref{sec:exp tight}. Finally, the proof of Theorems \ref{theorem MDP of current} and \ref{theorem MDP of tagged particle} is presented in Section \ref{section proofs}. To make the paper easier to follow, we put the tedious calculations in Section \ref{sec:calculations}. 

\section{Model and results}\label{sec:result}

\subsection{Model} The state space of the one-dimensional symmetric simple exclusion process (SSEP) is  $\Omega=\{0,1\}^{\Z}$. For a configuration $\eta \in \Omega$, $\eta (x) = 1$ if and only if there is one particle at site $x$.  The infinitesimal generator of the process acts on local functions $f: \Omega \rightarrow \R$ as
\begin{equation}\label{generator}
	\gen f (\eta) = \frac{1}{2} \sum_{x \in \Z}  \big\{f(\eta^{x,x+1})  -  f(\eta) \big\}.
\end{equation}
Above, we call $f$ a local function if it depends on $\eta$ only through a finite number of sites, and  $\eta^{x,y}$ is the configuration obtained from $\eta$ by exchanging the values of $\eta(x)$ and $\eta(y)$, that is,
\[\eta^{x,y} (z) = \begin{cases}
	\eta (z) \quad&\text{if\;} z \neq x,y;\\
	\eta (y) \quad&\text{if\;} z = x;\\
	\eta (x) \quad&\text{if\;} z = y.
\end{cases}\]
For each $\rho \in [0,1]$, let $\nu_\rho$ be the Bernoulli product measure on $\Omega$ with marginals given by
\[\nu_\rho \{\eta: \eta (x) = 1\} = \rho, \quad \forall x \in \Z.\]
It is well known that $\nu_\rho$ is reversible and ergodic for the SSEP, \emph{cf.}~\cite{liggettips} for example.

Denote by $\{\eta_t\}_{t \geq 0}$  the Markov process with generator $\gen$.  For a probability measure $\mu$ on $\Omega$, denote by $\P_\mu$ the measure on the path space $\mathcal{D} ([0,\infty), \Omega)$ of the process $\eta_t$ with initial measure $\mu$, and by $\E_\mu$ the corresponding expectation.

 \subsection{MDP for the current} We first recall basic facts of the fractional Brownian motion.  Let $\{B_t^{1/4}\}_{t\geq 0}$ be the fractional Brownian motion with Hurst parameter $1/4$, which is a mean-zero continuous-time Gaussian process with covariance given by 
\[
{\rm Cov}\left(B_t^{1/4}, B_s^{1/4}\right)=\frac{1}{2}\left(t^{1/2}+s^{1/2}-|t-s|^{1/2}\right) =: a (t,s)
\]
for any $s,t \geq 0$.  It is shown in \cite{Decreu1999FBM} that the fractional Brownian motion has the following representation,
\begin{equation}\label{equ 1.1}
	B_t^{1/4}=\int_0^t\mathcal{K}(t,s)d B_s, \quad \forall t \geq 0,
\end{equation}
where $\{B_t\}_{t \geq 0}$ is the standard one-dimensional Brownian motion starting from the origin, and the kernel function is explicitly given by
\[
\mathcal{K}(t,s)=\frac{(t-s)^{-1/4}}{\sqrt{V}\Gamma(3/4)}F(1/4, -1/4, 3/4, 1-\frac{t}{s})
\]
for any $0\leq s<t$, where
\[
V=\frac{8\Gamma(3/2)\cos(\pi/4)}{\pi}
\text{\quad and
\quad}
F(\alpha,\beta,\gamma,z)=\sum_{k=0}^{+\infty}\frac{(\alpha)_k(\beta)_k}{(\gamma)_kk!}z^k
\]
with $(a)_k$ defined as $\Gamma(a+k)/\Gamma(a)$.
Moreover, it is easy to see that  for any $s,t \geq 0$,
\[
\int_0^{t\wedge s}\mathcal{K}(t,\tau)\mathcal{K}(s,\tau)d\tau= a (t,s).
\]
Fix a time horizon $T > 0$. Let $\mathcal{H}$ be the set of c\`{a}dl\`{a}g functions $f:[0, T] \rightarrow \R$ such that there exists a function  $h_f \in L^2[0, T]$ satisfying
\[
f (t)=\int_0^t\mathcal{K}(t,s)\,h_f(s)ds, \quad \forall  0\leq t\leq T.
\]
For any $f\in\mathcal{D}([0,T],\R)$, define
\[
\mathcal{I}_{path}(f)=
\begin{cases}
	\frac{1}{2}\int_0^T(h_f(s))^2ds \quad&\text{if~}f\in \mathcal{H},\\
	+\infty \quad&\text{otherwise}. 
\end{cases}
\]
By \eqref{equ 1.1} and \cite[Theorem 3.4.12]{deuschel1989large},  $\mathcal{I}_{path}(f)$ is the large deviation rate function of the sequence of the processes $\{\frac{1}{\sqrt{N}}B_t^{1/4}:~0 \leq t \leq T\}_{N\geq 1}$.

 For each $x \in \bb{Z}$, let $J_{x,x+1} (t)$ be the net number of particles across the bond $(x,x+1)$ up to time $t$, \emph{i.e.}, the number of particles jumping from $x$ to $x+1$ minus the number of particles jumping from $x+1$ to $x$ during the time interval $[0,t]$.
 
The first result concerns moderate deviations for the path of the current when the process starts from the initial measure $\nu_\rho$ for some $\rho \in (0,1)$. Throughout the article, we assume $\{a_N\}_{N \geq 1}$ to be a sequence of positive numbers such that
\[\lim_{N \rightarrow \infty} \frac{a_N}{N} = \lim_{N \rightarrow \infty} \frac{N \log N}{a_N^2} = 0.\]

\begin{theorem}\label{theorem MDP of current}
Fix $\rho \in (0,1)$. For $f \in \mathcal{D}([0,T], \R)$, define
	\[
	\Ical_{cur}(f)=\frac{1}{\sigma_J^2}\mathcal{I}_{path}(f)
	\]
where  $\sigma_J^2 := \sigma_J^2 (\rho) :=  \sqrt{\frac{2}{\pi}}\rho(1-\rho)$. Then,  the sequence of the processes $\left\{\frac{1}{a_N}J_{-1,0}(tN^2): 0 \leq t \leq T\right\}_{N \geq 1}$ satisfies moderate deviation principles with rate  $a_N^2/N$ and with rate function $\Ical_{cur}$. Precisely speaking, 	for any closed set $\Ccal \subseteq \mathcal{D}([0,T], \R)$, 
	\[
	\limsup_{N\rightarrow+\infty}\frac{N}{a_N^2}\log \P_{\nu_\rho} \Big(\left\{\frac{1}{a_N}J_{-1,0}(tN^2): 0\leq t\leq T\right\}\in \Ccal \Big)\leq -\inf_{f\in \Ccal} \Ical_{cur}(f),
	\]
 and for any open set $\Ocal \subseteq \mathcal{D}([0,T], \R)$,
	\[
	\liminf_{N\rightarrow+\infty}\frac{N}{a_N^2}\log \P_{\nu_\rho} \Big(\left\{\frac{1}{a_N}J_{-1,0}(tN^2): 0\leq t\leq T\right\}\in \Ocal \Big)\geq -\inf_{f\in \Ocal} \Ical_{cur}(f).
	\]
\end{theorem}

\subsection{MDP for the tagged particle} Note that the particles in the SSEP are indistinguishable. In this subsection, we distinguish one particular particle, put it initially at the origin, and call it the \emph{tagged} particle. Denote by $X(t)$ the position of the tagged particle at time $t$. By convention, $X(0)=0$. Since particles cannot jump over each other, the relative ordering of the particles is conserved.

Observe that the tagged particle process $\{X (t)\}_{t \geq 0}$ itself is not Markovian. However, both the coupled process $\{(\eta_t,X(t))\}_{t \geq 0}$ and the environment process $\{\xi_t\}_{t \geq 0}$ are Markovian, where $\xi_t (x) = \eta_t (X (t) + x)$ for $x \in \Z$ and $t \geq 0$.

For $\rho \in (0,1)$, let $\nu_\rho^*$ be the measure obtained from  the Bernoulli product measure $\nu_\rho$ conditioned on having a particle at the origin, {\emph i.e.},
\[\nu_\rho^* (\cdot) = \nu_\rho (\cdot \;| \eta(0) = 1).\]
It is also well known that $\nu_\rho^*$ is invariant and ergodic for the environment process $\{\xi_t\}_{t \geq 0}$.

The second result of the article concerns moderate deviations for the sample path of the tagged particle when the process starts from the initial measure $\nu_\rho^*$ for some $\rho \in (0,1)$.

\begin{theorem}\label{theorem MDP of tagged particle}
Fix $\rho \in (0,1)$. For $f \in \mathcal{D}([0,T], \R)$, define
	\[
	\Ical_{tag}(f)=\frac{1}{\sigma_X^2}\mathcal{I}_{path}(f)
	\]
	where $\sigma_X^2 = \sqrt{\frac{2}{\pi}}\frac{1-\rho}{\rho}$.   Then,  the sequence of the tagged particle processes $\left\{\frac{1}{a_N} X (tN^2): 0\leq t\leq T\right\}_{N \geq 1}$ satisfies moderate deviation principles with rate  $a_N^2/N$ and with rate function $\Ical_{tag}$. Precisely speaking, 	for any closed set $\Ccal \subseteq \mathcal{D}([0,T], \R)$, 
	\[
	\limsup_{N\rightarrow+\infty}\frac{N}{a_N^2}\log \P_{\nu_\rho^*} \Big(\left\{\frac{1}{a_N}X(tN^2): 0\leq t\leq T \right\}\in \Ccal \Big)\leq -\inf_{f\in \Ccal}\Ical_{tag}(f),
	\]
 and for any open set $\Ocal \subseteq \mathcal{D}([0,T], \R)$,
	\[
	\liminf_{N\rightarrow+\infty}\frac{N}{a_N^2}\log \P_{\nu_\rho^*} \Big(\left\{\frac{1}{a_N}X(tN^2): 0\leq t\leq T\right\} \in \Ocal\Big)\geq -\inf_{f\in \Ocal}\Ical_{tag}(f).
	\]
\end{theorem}

\section{Preliminary estimates}\label{sec:estimates}

\subsection{Moderate deviations from hydrodynamic limits} We first recall moderate deviation principles for the empirical measure of the SSEP when the process starts from its invariant measure $\nu_\rho$ for some $\rho \in (0,1)$, see \cite{gao2003moderate} for details.  For each $N \geq 1$, the centered empirical measure $\mu^N_t \in \mathcal{S}^\prime(\bb{R})$ of the SSEP,  which acts on Schwartz functions $G \in \Scal(\R)$, is defined as
\[\<\mu^N_t,G\> = \frac{1}{a_N} \sum_{x \in \Z} \big( \eta_{tN^2} (x) - \rho \big) G (x/N).\]
In this subsection, we assume the sequence of positive numbers $\{a_N\}_{N \geq 1}$ satisfies
\[\lim_{N \rightarrow \infty} \frac{a_N}{N} = \lim_{N \rightarrow \infty} \frac{N}{a_N^2} = 0.\]
Let $\mathcal{D}([0,T],\mathcal{S}^\prime (\bb{R}))$ be the set of c\`{a}dl\`{a}g functions from $[0, T]$ to $\mathcal{S}^\prime (\bb{R})$  endowed with the topology under which $\lim_{n\rightarrow+\infty}\theta^n=\theta$ in $\mathcal{D}([0,T],\mathcal{S}^\prime (\bb{R}))$ if and only if \[\{\theta^n_t(G)\}_{0\leq t\leq T}\rightarrow \{\theta_t(G)\}_{0\leq t\leq T}\]
in $\mathcal{D}([0,T], \R)$ for any $G\in \Scal(\R)$.
For $G \in \Ccal^{1,\infty}_c ( [0,T]  \times \R)$ and $\mu \in \Dcal ([0,T],\mathcal{S}^\prime (\bb{R}))$,  define the linear functional $l (\cdot,\cdot)$ as
\[l(\mu, G)=\left\langle\mu_{T}, G_T\right\rangle-\left\langle\mu_{0}, G_0 \right\rangle-\int_{0}^{T}\langle\mu_{s},\left(\partial_s+(1/2) \partial_u^2\right) G_s \rangle d s.\]
Above, $G_s(u) = G(s,u)$ and $\langle \cdot,\cdot\rangle$ is the bilinear form on $\mathcal{S}^\prime(\bb{R}) \times \mathcal{S} (\bb{R})$ such that $\langle \mu,g\rangle=\mu(g)$ is the output of $\mu$ acting on $g$ for $\mu\in \mathcal{S}^\prime(\bb{R})$ and $g\in \mathcal{S} (\bb{R})$. The rate function $\mathcal{Q}_T (\mu) := \mathcal{Q}_{T,dyn} (\mu)+ \mathcal{Q}_0 (\mu_0)$ is defined as
\begin{equation*}
	\begin{aligned} \mathcal{Q}_{T,dyn}(\mu) &=\sup _{G \in \Ccal^{1,\infty}_c ( [0,T]  \times \R)}\left\{l(\mu, G)-\frac{\chi (\rho)}{2}\int_0^T  \int_\R (\partial_u G)^2 (s,u)\,du\,ds\right\}, \\ \mathcal{Q}_{0}\left(\mu_{0}\right) &=\sup _{\phi \in \Ccal_c^\infty \left(\R\right)}\left\{\left\langle\mu_{0}, \phi\right\rangle-\frac{\chi (\rho)}{2} \int_{\R} \phi^2 (u) d u\right\}, \end{aligned}
\end{equation*}
where $\chi (\rho) = \rho (1-\rho)$ is the compressibility of the SSEP.

\begin{theorem}[\cite{gao2003moderate}]\label{thm:mdphl}
	Let $\rho \in (0,1)$. The  sequence of processes $\{\mu^N_t: 0 \leq t \leq T\}_{N \geq 1}$ satisfies the MDP with decay rate $a_N^2 / N$ and with rate function $\mathcal{Q}_T (\mu)$. More precisely, for any closed set $\Ccal \subseteq \Dcal ([0,T], \mathcal{S}^\prime (\R))$,
	\[\limsup_{N \rightarrow \infty} \frac{N}{a_N^2} \P_{\nu_\rho} \big( \{\mu^N_t: 0 \leq t \leq T\} \in \Ccal \big) \leq - \inf_{\mu \in \Ccal} \mathcal{Q}_T (\mu),\] 
	and for any open set $\Ocal \subseteq \Dcal ([0,T], \mathcal{S}^\prime (\R))$,
	\[\liminf_{N \rightarrow \infty} \frac{N}{a_N^2} \P_{\nu_\rho} \big( \{\mu^N_t: 0 \leq t \leq T\} \in \Ocal \big) \geq - \inf_{\mu \in \Ocal} \mathcal{Q}_T (\mu).\]
\end{theorem}

\begin{remark}\label{rmk:coupling}
	Let $\eta_\cdot$ and $\eta_\cdot^*$ be two SSEPs with initial distributions $\nu_\rho$ and $\nu_\rho^*$ respectively. By the basic coupling (see \cite{liggettips} for example), one could couple $\eta_\cdot$ and $\eta_\cdot^*$ together such that $\eta_t (x) =\eta^*_t (x)$ for all but one $x$.  As a consequence, the associated empirical measures of the two processes satisfy
	\[|\<\mu^N_t - \mu^{N,*}_t, G\>| \leq \frac{\|G\|_\infty}{a_N}, \quad \forall t\geq 0, \;\forall G \in \Scal (\R).\]
	Thus, the above theorem also holds under $\P_{\nu_\rho^*}$.
\end{remark}

%\begin{remark}
%Reference \cite{gao2003moderate} proves Theorem \ref{thm:mdphl} under $\P=\P_{\nu_\rho}$. The proof relies on several estimations of some events having super-exponentially small probabilities with respect to the decay rate $\frac{a_N^2}{N}$. Note that
%\begin{equation}\label{equ relationship between two probability}
%\P_{\nu_\rho^*}(\cdot)\leq \rho^{-1}\P_{\nu_\rho}(\cdot),
%\end{equation}
%hence all super-exponential estimations given in \cite{gao2003moderate} hold under $\P_{\nu_\rho^*}$ and consequently Theorem \ref{thm:mdphl} holds under $\P=\P_{\nu_\rho^*}$.
%\end{remark}

\subsection{Properties}  In this subsection, we establish several properties of the measure $\mu$ when it satisfies $\Qcal_T (\mu) < \infty$.  Before that, we first introduce some notation. For $H,G \in \Ccal^{1,2}_c ([0,T] \times \R)$, define
\[[H,G] := \int_{0}^T \int_{\R} \partial_{u} H(t,u) \partial_{u}G(t,u)\,du\,dt.\]
In order to make $[\cdot,\cdot]$ an inner product, we say $H \sim G$ if $[H-G,H-G] = 0$. Let $\mathcal{H}^1$ be the Hilbert space obtained as the completion of  $\Ccal^{1,2}_c ([0,T] \times \R)) / \sim$ with respect to the inner product $[\cdot,\cdot]$. 

\begin{lemma}\label{lem:mu property-1}
Let $\mu \in \Scal^\prime (\R)$. If $\Qcal_T (\mu) < \infty$, then there exist functions $\psi \in L^2 (\R)$ and $H \in \mathcal{H}^1$ such that, for any $\phi \in \Scal (\R)$,
\[\<\mu_0,\phi\> = \int_{\R} \psi (u) \phi (u) du,\]
and that $\mu$ is the unique weak solution of the following PDE
\begin{equation}\label{mu PDE}
	\begin{cases}
		\partial_t \mu (t,u) = (1/2) \partial_u^2 \mu (t,u) - \chi (\rho) \partial_u^2 H (t,u), \quad t > 0,\; u \in \R,\\
		\mu (0, u) = \psi (u), \quad u \in \R.
	\end{cases}
\end{equation}
In particular, $\mu = \mu (t,u)$ is a function. Moreover,
\[\mathcal{Q}_0 (\mu_0) = \frac{\|\psi\|_{L^2 (\R)}^2}{2 \chi(\rho)}, \quad 	\mathcal{Q}_{T,dyn}(\mu) = \frac{\chi(\rho)}{2} [H,H].\]
\end{lemma}

The proof of the above lemma follows directly from Riesz's representation theorem.  We refer the readers to Appendix \ref{app:proof lem mu p-1} for its proof. 

In the above lemma, we say $\mu$ is a weak solution to \eqref{mu PDE} if for any $G \in \Ccal_c^{1,2} ([0,T] \times \R)$ and for any $0 < t \leq T$, 
\begin{multline*}
	\int_{\R} \mu (t,u) G (t,u) du = 	\int_{\R} \psi (u) G (0,u) du + \int_{0}^t \int_{\R} \mu (s,u) (\partial_{s} + (1/2) \partial_{u}^2) G (s,u)\,du\,ds \\
	+ \chi (\rho) \int_{0}^t \int_{\R} \partial_{u} G (s,u) \partial_{u} H(s,u)\,du\,ds.
\end{multline*}
Actually, $\mu$ has the following explicit expression: for any $t \geq 0$ and for any $u \in \R$,
\begin{equation}\label{mu formula}
	\mu(t,u) = \int_{\R} p_t (u-v) \psi (v) dv - \chi (\rho) \int_0^t \int_{\R} p_{t-s}^\prime (u-v) \partial_v H (s,v) \,dv \,ds,
\end{equation}
where $p_t (u)$ is the heat kernel
\[p_t (u)  = \frac{1}{\sqrt{2 \pi t}} e^{-u^2/(2t)}\]
and $p^\prime_t(u):=\frac{\partial}{\partial u}p_t(u)$. Using the explicit formula of $\mu$, we have the following lemma.

\begin{lemma}\label{lem:mu property-2}
 If $\Qcal_T (\mu) < \infty$, then the limit
 \[\int_{0}^\infty [\mu(t,u) - \mu (0,u)] du := \lim_{M \rightarrow \infty} \int_{0}^M [\mu(t,u) - \mu (0,u)] du\]
 is well-defined. Moreover,
 \[\int_{0}^\infty [\mu(t,u) - \mu (0,u)] du  =  \int_0^t  \int_{\R}  \Big[-\frac{1}{2} {p}_s^{\prime} (v) \psi (v) + \chi (\rho) p_{t-s}(v) \partial_v H (s,v) \Big]  \,dv \,ds. \]
\end{lemma}

\begin{proof}
By \eqref{mu formula}, 
\[\int_{0}^M [\mu(t,u) - \mu (0,u)] du = {\rm I} - {\rm II},\]
where 
\begin{align*}
{\rm I} &= \int_0^M \int_{\R} \big[p_t (u-v) - p_0 (u-v)\big]\psi (v) \,dv\,du,\\ 
{\rm II} &= \chi (\rho) \int_0^t \int_0^M \int_{\R} p_{t-s}^\prime (u-v) \partial_v H (s,v) \,dv\,du \,ds.
\end{align*}
For the first term, we write it as
\begin{multline*}
{\rm I} = \int_0^M \int_{\R}  \int_0^t \dot{p}_s (u-v) \psi(v)\,ds\,dv\,du = \frac{1}{2} \int_0^M \int_{\R}  \int_0^t {p}_s^{\prime\prime} (u-v) \psi(v)\,ds\,dv\,du \\
= \frac{1}{2}  \int_{\R}  \int_0^t \big[{p}_s^{\prime} (M-v) -{p}_s^{\prime} (v)\big] \psi(v)\,ds\,dv.
\end{multline*}
Above, $\dot{p}_s (u) := \frac{\partial}{\partial s} p_s (u) $ and we used the identity $\dot{p}_s (u)= (1/2) p_s^{\prime\prime} (u)$. For the second term,
\[{\rm II} = \chi (\rho) \int_0^t  \int_{\R} \big[p_{t-s} (M-v) - p_{t-s}(v) \big]\partial_v H (s,v) \,dv \,ds.\]
Thus, we conclude the proof once we show that
\begin{align}
	\lim_{M \rightarrow \infty} \int_{\R}  \Big(\int_0^t {p}_s^{\prime} (M-v)\,ds \Big) \psi(v) \,dv = 0,\label{pre1}\\
	\lim_{M \rightarrow \infty}  \int_0^t  \int_{\R} p_{t-s} (M-v) \partial_v H (s,v) \,dv \,ds = 0.\label{pre2}
\end{align}
One could check directly that, for any $t > 0$, $p_s (u) \in L^2 ([0,t] \times \R)$ and 
\begin{equation}\label{pre3}
\int_0^t {p}_s^{\prime} (\cdot)\,ds  \in L^2 (\R).
\end{equation}
We refer the readers to Appendix \ref{app pf pre3} for the proof of \eqref{pre3}. Finally, \eqref{pre1} and \eqref{pre2} follow from the following claim: if $f,g \in L^2 (\R)$, then
\[\lim_{M \rightarrow \infty} \int_{\R} f(M-u) g (u) du = 0.\]
The similar result holds when $f,g \in L^2 ([0,t] \times \R)$. Indeed, by splitting
\[\int_{\R} f(M-u) g (u) du = \int_{|u| > M/2} f(M-u) g (u) du + \int_{|u| \leq M/2} f(M-u) g (u) du,\]
and by Cauchy-Schwarz inequality, we bound
\[\Big|\int_{\R} f(M-u) g (u) du\Big| \leq \|f\|_{L^2 (\R)} \Big(\int_{|u| > M/2}  g (u)^2 du \Big)^{1/2} + \|g\|_{L^2 (\R)} \Big(\int_{u > M/2}  f (u)^2 du \Big)^{1/2},\]
which converges to zero as $M \rightarrow \infty$. This concludes the proof.
\end{proof}

\begin{remark}\label{rmk:mu property}
By repeating the above arguments, one could also show that if $\Qcal_T (\mu) < \infty$, then  for any $0 < t \leq T$,
\[\lim_{M \rightarrow \infty} \frac{1}{M} \int_{0}^M u [\mu(t,u) - \mu (0,u)] \,du=0.\]
\end{remark}

\begin{remark}\label{rmk:goodness of Qcal}
Using \eqref{mu formula} and Lemma \ref{lem:mu property-1}, it is not difficult to show that the moderate deviation rate function $\Qcal_T (\cdot)$ is good. Here we give an outline of the check of this property. The property that $\Qcal_T (\cdot)$ is lower semi-continuous follows directly from the fact that $\Qcal_T (\cdot)$ is the supremum of a class of continuous functions from $\Dcal ([0,T], \mathcal{S}^\prime (\R))$ to $\R$. Consequently, 
\[
\left\{\mu:~\Qcal_T (\mu)\leq c\right\}:=\mathbb{L}_c
\]
is closed for any $c>0$. Hence, we only need to show that $\mathbb{L}_c$ is compact. Using \eqref{mu formula}, Lemma \ref{lem:mu property-1}, the formula of integration by parts, Cauchy-Schwarz inequality and $\partial_tp(t,u)=\frac{1}{2}\partial^2_up(t,u)$, we have
\begin{align*}
\left|\<\mu_t, f\>\right|&\leq \|\psi\|_{L^2(\R)}\|f\|_{L^2(\R)}+\chi(\rho)\sqrt{T}\|f^\prime\|_{L^2(\R)}\sqrt{[H, H]}\\
&=\sqrt{2\chi(\rho)\mathcal{Q}_0(\mu_0)}\|f\|_{L^2(\R)}+\sqrt{T}\|f^\prime\|_{L^2(\R)}\sqrt{2\chi(\rho)\mathcal{Q}_{T,dyn}(\mu)}\\
&\leq 2 \sqrt{\chi(\rho)\Qcal_T (\mu)}\left(\|f\|_{L^2(\R)}+\sqrt{T}\|f^\prime\|_{L^2(\R)}\right)
\end{align*}
and
\begin{align*}
&\left|\<\mu_t, f\>-\<\mu_s, f\>\right|\\
&\leq \frac{t-s}{2}\|f^{\prime\prime}\|_{L^2(\R)}\|\psi\|_{L^2(\R)}+\chi(\rho)\sqrt{t-s}\|f^{\prime}\|_{L^2(\R)}\sqrt{[H, H]}
+\frac{\chi(\rho)}{2}(t-s)\sqrt{T}\|f^{\prime\prime\prime}\|_{L^2(\R)}\sqrt{[H, H]}\\
&\leq \sqrt{2\Qcal_T (\mu)\chi(\rho)}\left(\frac{t-s}{2}\|f^{\prime\prime}\|_{L^2(\R)}+\sqrt{t-s}\|f^\prime\|_{L^2(\R)}+\frac{t-s}{2}\sqrt{T}\|f^{\prime\prime\prime}\|_{L^2(\R)}\right)
\end{align*}
for any $\mu$ such that $\Qcal_T (\mu)<+\infty$, any $f\in \mathcal{S} (\R)$ and any $0\leq s\leq t\leq T$. Therefore, for any $f\in \mathcal{S}(\R)$, $\{\mu_t(f):~0\leq t\leq T\}_{\mu\in \mathbb{L}_c}$ are uniformly bounded and equicontinuous. As a result, the compactness of $\mathbb{L}_c$ follows from Ascoli-Arzela theorem.  

\end{remark}

\subsection{Alternative formulas for $\Qcal_T$}   Introduce 
\[\mathcal{A} = \big\{ \mu: \text{there exist}\;H \in \Ccal^{1,2}_c ([0,T] \times \R)) \;\text{and}\; \psi \in \Ccal^2_c (\R) \;\text{such that $\mu$ satisfies\,\eqref{mu PDE}} \big\}.\]
One could check directly that if $\mu \in \mathcal{A}$, then for any $t > 0$, $\mu(t,\cdot)$ is integrable on $\R$ and vanishes at infinity.
For $\mu \in \Acal$, define $J$ the macroscopic instantaneous current as 
\begin{equation}\label{J}
J = -\frac{1}{2} \partial_{u} \mu + \chi (\rho) \partial_{u} H.
\end{equation}
 Obviously, 
\[\partial_t \mu + \partial_{u} J = 0.\] 
Define $K(t,u)$ the macroscopic current across the macroscopic point $u$ during time $[0,t]$ as
\[K(t,u) := \int_0^t J(s,u) ds.\]
Then, for $u \in \R$ and $t > 0$, 
\begin{equation}\label{mu_K}
	\partial_{u} K(t,u) = \mu (0,u) - \mu (t,u).
\end{equation}
By direct calculations, for $\mu \in \Acal$,
\begin{multline}\label{q_mu}
\chi (\rho)  \Qcal_T (\mu) = 	\chi (\rho) \Qcal_T  (\mu_0,K) := \frac{1}{2} \int_{0}^{T} \int_{\mathbb{R}}\left(\partial_{t} K\right)^{2}(t, u)  d u d t+\frac{1}{4} \int_{\mathbb{R}} \left(\partial_{u} K\right)^{2}(T, u) d u 
	\\ +\frac{1}{8} \int_{0}^{T} \int_{\mathbb{R}}\left(\partial_{u} \mu_{0}(u)- \partial_{u}^{2} K(t, u) \right)^{2} d u d t+\frac{1}{2} \int_{\mathbb{R}} \big[ \mu_{0}^{2}(u)- \mu_{0}(u) \partial_{u} K(T, u) \big]d u.
\end{multline}
The above formula is exactly \cite[Eqn. (3.8)]{xue2023moderate}, whose proof is presented  in Appendix \ref{appendix:pf_q_mu}.

\section{Variational formula for the rate function}\label{sec:variation}

In this section, we present a variational formula for the rate function $\mathcal{I}_{path}$. For any $n \geq 1$, any column vector $\alphabf \in \R^n$ and any $0\leq t_1<t_2<\ldots<t_n$, let
\begin{equation}\label{I_tj_n}
\mathcal{I}_{_{\{t_j\}_{j=1}^n}}(\alphabf)=\frac{1}{2}\alphabf^TA^{-1}_{_{\{t_j\}_{j=1}^n}}\alphabf
\end{equation}
where $A_{_{\{t_j\}_{j=1}^n}}$ is the $n\times n$ matrix with  $(k,\ell)$-component equal to $a(t_k,t_\ell)$ for $1 \leq k,\ell \leq n$. Recall $a(\cdot,\cdot)$ is the covariance function of the fractional Brownian motion with Hurst parameter $1/4$.

The following variational formula relates $\mathcal{I}_{path}$ the rate function of the sample path to the corresponding finite-dimensional rate functions, which is a direct consequence of \cite[Theorem 4.28]{Feng2006LDP}.

\begin{proposition}\label{prop:rateF_varF_1}
	For any $f\in \mathcal{D}([0,T],\R)$,
	\[
	\mathcal{I}_{path}(f)=\sup\Big\{\mathcal{I}_{_{\{t_j\}_{j=1}^n}}\left(f_{_{\{t_j\}_{j=1}^n}}\right):~n\geq 1, 0\leq t_1<t_2<\ldots<t_n\leq T, t_j\in \Delta_f^c \text{~for all~}1\leq j\leq n\Big\},
	\]
	where $f_{_{\{t_j\}_{j=1}^n}}=(f ({t_1}), f ({t_2}),\ldots,f ({t_n}))^T$ and $\Delta_f^c$ is the set of continuities of $f$.
\end{proposition}

\begin{proof}
	It is well known that for any $0\leq t_1<t_2<\ldots<t_n$, $\mathcal{I}_{_{\{t_j\}_{j=1}^n}}(\cdot)$ is the large deviation rate function of the sequence of Gaussian vectors $\left\{\frac{1}{\sqrt{N}}\left(B^{1/4}_{t_1}, B^{1/4}_{t_2}, \ldots, B^{1/4}_{t_n}\right)\right\}_{N\geq 1}$, and $\mathcal{I}_{path} (\cdot)$ is the rate function of the sequence $\left\{\frac{1}{\sqrt{N}}B_t^{1/4}:~0 \leq t \leq T\right\}_{N\geq 1}$, see \cite{deuschel1989large} for example.  Then, by \cite[Theorem 4.28]{Feng2006LDP}, the result follows immediately.
\end{proof}

The following lemma gives a variational formula for the finite-dimensional rate function appearing in the last proposition, which is the main result of this section.

\begin{proposition}\label{prop:rateF_varF_2}
	For any positive integer $n$, any column vector $\alphabf = (\alpha_1,\ldots,\alpha_n)^T \in \R^n$, and any $0 < t_1 < t_2 < \ldots < t_n \leq T$,  we have 
	\begin{equation}\label{Q_n}
		\inf \Big\{\Qcal_T (\mu): \int_0^\infty [\mu(t_j,u) - \mu(0,u)] du= \alpha_j,\; 1 \leq j \leq n \Big\} = \frac{\sqrt{2\pi}}{4\chi(\rho)} \alphabf^TA^{-1} \alphabf,
	\end{equation}
	where $A = A_{_{\{t_j\}_{j=1}^n}}$ for short.
\end{proposition}

Note that 
\[\frac{\sqrt{2\pi}}{4\chi(\rho)} \alphabf^TA^{-1} \alphabf = \frac{1}{\sigma_J^2} \mathcal{I}_{_{\{t_j\}_{j=1}^n}}(\alphabf).\]
In the rest of this section, we prove Proposition \ref{prop:rateF_varF_2}.

\subsection{Restricted to $\mu \in \Acal$} We first show that we could restrict the infimum in Proposition \ref{prop:rateF_varF_2} to $\mu \in \Acal$, which is the content of the next result.

\begin{proposition}[The case $n > 1$]\label{prop:Q_finiteD}
	For any positive integer $n$, any column vector $\alphabf \in \R^n$, and any $0 < t_1 < t_2 < \ldots < t_n \leq T$,  we have 
	\begin{equation}\label{Q_n_A}
		\inf \Big\{\Qcal_T (\mu): \mu \in \Acal, \;\int_0^\infty [\mu(t_j,u) - \mu(0,u)] du= \alpha_j,\; 1 \leq j \leq n \Big\} = \frac{\sqrt{2\pi}}{4\chi(\rho)} \alphabf^TA^{-1} \alphabf,
	\end{equation}
	where $A = A_{_{\{t_j\}_{j=1}^n}}$.
\end{proposition}

We first prove Proposition \ref{prop:rateF_varF_2} from the last result, and then prove the last proposition in the next subsection. 

\begin{proof}[Proof of Proposition \ref{prop:rateF_varF_2}]
By Proposition \ref{prop:Q_finiteD},
\[		\inf \Big\{\Qcal_T (\mu): \int_0^\infty [\mu(t_j,u) - \mu(0,u)] du= \alpha_j,\; 1 \leq j \leq n \Big\}  \leq \frac{\sqrt{2\pi}}{4\chi(\rho)} \alphabf^TA^{-1} \alphabf,\]
and it remains to prove the opposite direction of the last inequality. 
Take $\mu$ such that $\Qcal_T (\mu) < \infty$  and 
\[\int_0^\infty [\mu(t_j,u) - \mu(0,u)] du= \alpha_j,\; \forall 1 \leq j \leq n.\]
Then, by Lemma \ref{lem:mu property-1}, there exist $\psi \in L^2 (\R)$ and $H \in \mathcal{H}^1$ such that $\mu$ is the unique weak solution to the PDE \eqref{mu PDE}. For $\varepsilon > 0$, let $\psi^\varepsilon \in \Ccal^2_c (\R)$ and $H^\varepsilon \in \Ccal_c^{1,2} ([0,T] \times \R)$ such that
\[\psi^{\varepsilon} \rightarrow \psi\quad \text{in}\; L^2 (\R), \quad \text{and}\quad H^\varepsilon \rightarrow H\quad \text{in}\; \mathcal{H}^1, \quad \;\varepsilon \rightarrow 0.\]
Denote by $\mu^\varepsilon$ the solution of the PDE \eqref{mu PDE} associated with $\psi^\varepsilon$ and $H^\varepsilon$.  Finally, let 
\[ \alpha_j^\varepsilon := \int_0^\infty [\mu^\varepsilon(t_j,u) - \mu^\varepsilon (0,u)] du, \quad \forall 1 \leq j \leq n.\]
By Lemmas \ref{lem:mu property-1} and \ref{lem:mu property-2},
\[\lim_{\varepsilon \rightarrow 0} \Qcal_T (\mu^\varepsilon) = \Qcal_T (\mu), \quad \lim_{\varepsilon \rightarrow 0}  \alpha_j^\varepsilon = \alpha_j, \quad \forall 1 \leq j \leq n.\]
Thus, 
\begin{multline*}
\Qcal_T (\mu) = \lim_{\varepsilon \rightarrow 0} \Qcal_T (\mu^\varepsilon)\\
  \geq \limsup_{\varepsilon \rightarrow 0} 	\inf \Big\{\Qcal_T (\mu): \mu \in \Acal, \;\int_0^\infty [\mu(t_j,u) - \mu(0,u)] du= \alpha^\varepsilon_j,\; 1 \leq j \leq n \Big\} \\
  = \limsup_{\varepsilon \rightarrow 0}  \frac{\sqrt{2\pi}}{4\chi(\rho)} \alphabf_\varepsilon^TA^{-1} \alphabf_\varepsilon = \frac{\sqrt{2\pi}}{4\chi(\rho)} \alphabf^TA^{-1} \alphabf,
\end{multline*}
where $\alphabf_\varepsilon = (\alpha_1^\varepsilon,\alpha_2^\varepsilon, \ldots, \alpha_n^\varepsilon)^T$. We conclude the proof by taking the infimum over $\mu$ on the left-hand side of the last inequality.
\end{proof}

\subsection{Proof of Proposition \ref{prop:Q_finiteD}} The proof of Proposition \ref{prop:Q_finiteD} is based on Fourier approach and is divided into several lemmas. Observe that if $\mu \in \Acal$, then by \eqref{mu_K},
\[K(t,0) = \int_0^\infty [\mu(t,u) - \mu (0,u)] du, \quad t > 0,\]
which allows us to reduce the constraints inside the infimum in \eqref{Q_n_A} to $K(t_j,0) = \alpha_j$ for $1 \leq j \leq n$. Obviously, we also have $K(0,u) = 0$ for any $u \in \R$ by the definition of $K(t,u)$.

First, we consider the case $n=1$ and $K(T,0) = \alpha$ for some $\alpha \in \R$. The following result was proved in \cite{xue2023moderate} by constructing the minimizer of the variational problem directly. Here, we present a different approach, which allows us to generalize it to the case $n > 1$. 

\begin{lemma}\label{lem:Q_1} Fix $\alpha \in \R$.  Then,
	\begin{equation}\label{Q_0}
		\inf \{\Qcal_T (\mu): \mu \in \Acal, \;K(0,\cdot) = 0, \;K (T,0)=\alpha\} = \frac{\sqrt{2\pi} \alpha^2}{4 \chi (\rho) \sqrt{T}}. 
	\end{equation}
\end{lemma}

\begin{proof} The proof is divided into several steps:
	
	\medspace
	
	\emph{Step 1.} We first minimize $\Qcal_T  (\mu_0,K)$ subject to $K(0,\cdot) = 0$ and $K(T,\cdot)$ being a known function. We claim that  the infimum is attained at the point $(K^{T,\alpha} ,\mu^{T,\alpha} _0)$ such that
	\begin{align}
		- \partial_{t}^2 K^{T,\alpha}(t,u) + \frac{1}{4} \partial_{u}^4 K^{T,\alpha} (t,u) - \frac{1}{4} \partial_{u}^3 \mu^{T,\alpha} _0 (u) = 0,\label{K}\\
		-\frac{T}{4} \partial_{u}^2 \mu^{T,\alpha} _0 (u) + \frac{1}{4} \int_0^T \partial_{u}^3 K^{T,\alpha} (t,u) dt + \mu^{T,\alpha} _0 (u) - \frac{1}{2} \partial_{u} K^{T,\alpha} (T,u) = 0.\label{mu_0}
	\end{align}
Indeed, the above equations are obtained by solving, for any $\mu_0$ and $K$ such that $K(0,\cdot) = K(T,\cdot) = 0$, 
\[\frac{d}{d \varepsilon} \Qcal_T (\mu_{0}^{T,\alpha} + \varepsilon \mu_0, K^{T,\alpha} )|_{\varepsilon = 0} = \frac{d}{d \varepsilon} \Qcal_T (\mu_{0}^{T,\alpha}, K^{T,\alpha} + \varepsilon K )|_{\varepsilon = 0}=0.\]

	For any function $f \in \Scal (\R)$, define the Fourier transform of $f$ as
	\[\Fcal f (\xi) = \frac{1}{\sqrt{2\pi}} \int_{\R} e^{iu\xi} f(u) du, \quad \xi \in \R.\]
	Then, $(\Fcal f') (\xi) = -i\xi \Fcal f (\xi), \;\xi \in \R$.   This permits us to  rewrite \eqref{K} and \eqref{mu_0} as
	\begin{align}
		- \partial_{t}^2 \Fcal K^{T,\alpha}(t,\xi) + \frac{\xi^4}{4} \Fcal K^{T,\alpha} (t,\xi) - \frac{i \xi^3}{4} \Fcal \mu^{T,\alpha} _0 (\xi) = 0,\label{K_fourier}\\
		\Big\{1+\frac{T\xi^2}{4}\Big\} \Fcal \mu^{T,\alpha} _0 (\xi)= - i \Big\{ \frac{\xi}{2} \Fcal K^{T,\alpha} (T,\xi) + \frac{\xi^3}{4} \int_0^T  \Fcal K^{T,\alpha} (t,\xi) dt \Big\}.\label{mu_0_fourier}
	\end{align}
	Solving \eqref{K_fourier}, we have
	\begin{equation}\label{K_fourier_1}
		\Fcal K^{T,\alpha} (t,\xi)  = C_1 (\xi) e^{t\xi^2/2} + C_2 (\xi) e^{-t\xi^2/2} + \frac{i \Fcal \mu^{T,\alpha} _0 (\xi)}{\xi}
	\end{equation}
	for two functions $C_1 (\xi)$ and $C_2 (\xi)$, which are determined by the boundary conditions at time $0$ and $T$,
	\begin{equation}\label{boundary_condition}
		\begin{cases}
			C_1 (\xi) + C_2 (\xi) +  \frac{i \Fcal \mu^{T,\alpha} _0 (\xi)}{\xi} = 0,\\
			C_1 (\xi) e^{T\xi^2/2} + C_2 (\xi) e^{-T\xi^2/2} + \frac{i \Fcal \mu^{T,\alpha} _0 (\xi)}{\xi} = \Fcal K^{T,\alpha}  (T,\xi).
		\end{cases}
	\end{equation}
	Inserting \eqref{K_fourier_1} into \eqref{mu_0_fourier}, we have
	\begin{equation}\label{mu_0_fourier_1}
		\Fcal \mu^{T,\alpha} _0 (\xi) = - \frac{i \xi}{2} \Big\{   \Fcal K^{T,\alpha} (T,\xi) + C_1 (\xi) \big[e^{T\xi^2/2} - 1\big] + C_2 (\xi) \big[1- e^{-T\xi^2/2} \big] \Big\}.
	\end{equation}
	Finally, using \eqref{mu_0_fourier_1}, and solving \eqref{boundary_condition}, we have
	\begin{equation*}
		\begin{cases}
			C_1 (\xi) = \frac{1}{2(e^{T\xi^2/2} - 1)} \Fcal K^{T,\alpha}(T,\xi),\\
			C_2 (\xi) = - \frac{e^{T\xi^2/2}}{2(e^{T\xi^2/2} - 1)} \Fcal K^{T,\alpha} (T,\xi).
		\end{cases}
	\end{equation*}
	In conclusion, we have shown that, for $0 \leq t \leq T$,
	\begin{align}
		\Fcal \mu^{T,\alpha} _0 (\xi) &= - \frac{i \xi}{2}  \Fcal K^{T,\alpha}(T,\xi),\label{mu_0_fourier_2}\\
		\Fcal K^{T,\alpha}(t,\xi) &= \frac{1}{2} \Big[ \frac{e^{t\xi^2/2} - e^{(T-t)\xi^2/2}}{e^{T\xi^2/2} - 1} +1 \Big] \Fcal K^{T,\alpha}(T,\xi).\label{K_fourier_2}
	\end{align}

	\medspace
	
	\emph{Step 2.}  By Parseval's identity,
	\begin{multline}\label{parseval}
		\chi (\rho) \Qcal_T (\mu_0,K) = \int_0^T \int_{\R} \Big\{ \frac{1}{2}  \big|\partial_{t} \Fcal K (t,\xi)\big|^2 
		+ \frac{1}{8} \big|- i \xi \Fcal \mu_0 (\xi) + \xi^2  \Fcal K (t,\xi) \big|^2 \Big\} d\xi\,dt\\
		+ \frac{1}{4} \int_{\R} \Big\{ \big| \Fcal \mu_0 (\xi) \big|^2 + \big| \Fcal \mu_0 (\xi)  + i \xi  \Fcal K(T,\xi) \big|^2 \Big\} d \xi.
	\end{multline}
	Inserting  $(K^{T,\alpha},\mu^{T,\alpha} _0)$ into the last expression, and by \eqref{mu_0_fourier_2} and \eqref{K_fourier_2}, 
	\begin{multline}\label{Q_2}
		\chi (\rho) \Qcal_T (\mu^{T,\alpha} _0,K^{T,\alpha}) \\ = \int_0^T \int_{\R} \frac{\xi^4}{32[e^{T\xi^2/2} - 1]^2} |\Fcal K^{T,\alpha} (T,\xi)|^2 
	 \Big[\Big(e^{t\xi^2/2} + e^{(T-t)\xi^2/2}\Big)^2 + \Big(e^{t\xi^2/2} - e^{(T-t)\xi^2/2}\Big)^2\Big] \,d\xi\,dt \\
		+  \int_{\R} \frac{\xi^2}{8} |\Fcal K^{T,\alpha} (T,\xi)|^2\,d\xi
		= \int_{\R} \frac{\xi^2e^{T\xi^2/2}}{4[e^{T\xi^2/2} - 1]}  |\Fcal K^{T,\alpha} (T,\xi)|^2\,d\xi. 
	\end{multline}

	\medspace
	
	\emph{Step 3.} Finally, we need to optimize $\Qcal_T (\mu^{T,\alpha} _0,K^{T,\alpha})$ over $\Fcal K^{T,\alpha} (T,\xi)$. The boundary condition  $K^{T,\alpha} (T,0) = \alpha$ is equivalent to
	\begin{equation}\label{boundary_condition_1}
		\frac{1}{\sqrt{2\pi}} \int_\R \Fcal K^{T,\alpha}(T,\xi) d \xi = \alpha.
	\end{equation}
	Subject to \eqref{boundary_condition_1}, the infimum of  \eqref{Q_2} is attained at the point 
	\[\Fcal K^{T,\alpha}(T,\xi) = \frac{\alpha (1-e^{-T \xi^2/2})}{\sqrt{T} \xi^2},\]
	and equals
	\[(\sqrt{2\pi} \alpha)^2 \Big(\int_{\R} \frac{4 (e^{T\xi^2/2} - 1)}{\xi^2 e^{T\xi^2/2}} \,d\xi\Big)^{-1} = \frac{\sqrt{2\pi} \alpha^2}{4 \sqrt{T}}\]
	by \eqref{integral_formula}. This concludes the proof.
\end{proof}

The next lemma extends the  last result to the case $n = 1$ and $K(t,0) = \alpha$ for $0 < t \leq T$ and $\alpha \in \R$.

\begin{lemma}\label{lem:Q_2} Fix $0 < t \leq T$ and $\alpha \in \R$. Then,
	\begin{equation}\label{Q_1}
		\inf \{\Qcal_T (\mu): \mu \in \Acal, \;K(0,\cdot) = 0, \;K (t,0)=\alpha\} = \frac{\sqrt{2\pi} \alpha^2}{4 \chi (\rho) \sqrt{t}}.
	\end{equation}
\end{lemma}

\begin{proof}
	For any $t > 0$, let $\mu^{t,\alpha}_T$  be the minimizer of \eqref{Q_1}, which does exist by the following construction.    
	Denote by $H^{t,\alpha}_T$ the function $H$ in \eqref{mu PDE}, respectively $K^{t,\alpha}_T$ the function $K$ in  \eqref{mu_K}, associated with the function $\mu^{t,\alpha}_T$.
	We call $\mu^{t,\alpha}_T$,  or correspondingly $(K_T^{t,\alpha},\mu^{t,\alpha}_{T,0})$, \emph{the minimizer of the rate function $\Qcal_T (\mu)$ with current $\alpha$  at time $t$.} In particular, when $T=t$, 
	\[\mu^{T,\alpha} = \mu^{T,\alpha}_T, \quad  H^{T,\alpha}= H^{T,\alpha}_T, \quad   K^{T,\alpha} = K^{T,\alpha}_T,\]
the precise expressions of which are given through the proof of Lemma \ref{lem:Q_1}.  
 
	We claim the infimum in \eqref{Q_1} is attained at the point $\mu_{T}^{t,\alpha}$ such that
	\begin{enumerate}[(i)]
		\item for $0 \leq s \leq t$, the function $\mu_{T}^{t,\alpha}$ is the unique solution to
		\begin{equation}\label{mu_T_t_alpha_1}
			\begin{cases}
				\partial_s \mu_{T}^{t,\alpha} (s,u) = \frac{1}{2} \partial_u^2 \mu_{T}^{t,\alpha} (s,u) - \chi(\rho) \partial_{u}^2 H^{t,\alpha} (s,u), 		\quad u \in \R,\\
				\mu_T^{t,\alpha} (0,u) = \mu^{t,\alpha}_0 (u), 		\quad u \in \R,
			\end{cases}
		\end{equation}
		where $H^{t,\alpha}$ is the function $H$ associated with the minimizer of the rate function $\Qcal_t (\mu)$ with current $\alpha$  at time $t$;
		\item for $t \leq s \leq T$, 
		the function $\mu_{T}^{t,\alpha}$ is the unique solution to
		\begin{equation}\label{mu_T_t_alpha_2}
			\begin{cases}
				\partial_s \mu_{T}^{t,\alpha} (s,u) = \frac{1}{2} \partial_u^2 \mu_{T}^{t,\alpha} (s,u), 		\quad u \in \R,\\
				\mu_T^{t,\alpha} (t,u) = \mu^{t,\alpha} (t,u), 		\quad u \in \R.
			\end{cases}
		\end{equation}
	\end{enumerate}
	Roughly speaking, the minimizer $\mu_{T}^{t,\alpha}$ evolves as  the minimizer of the rate function $\Qcal_t (\mu)$ with current $\alpha$  at time  $t$ during time interval $[0,t]$, and then evolves according to the hydrodynamic equation of the SSEP after time $t$.   This claim is obvious since
	\[	\mathcal{Q}_{T,dyn}(\mu) = \frac{\chi(\rho)}{2} [H,H],\]
	and thus 
	\[	\inf \{\Qcal_T (\mu): \mu \in \Acal, \;K(0,\cdot) = 0, \;K (t,0)=\alpha\} \geq 	\inf \{\Qcal_t (\mu): \mu \in \Acal, \;K(0,\cdot) = 0, \;K (t,0)=\alpha\}\]
 Moreover, by Lemma \ref{lem:Q_1}, the left-hand side of \eqref{Q_1} equals
	\[  \Qcal_T (\mu_T^{t,\alpha}) =\Qcal_t (\mu^{t,\alpha})+\int_t^T0ds=\frac{\sqrt{2\pi}\alpha^2}{4 \chi(\rho) \sqrt{t}},\]
	which concludes the proof.
\end{proof}

One could check directly that the functions $\mu_{T}^{t,\alpha}$ and $(K_T^{t,\alpha},\mu^{t,\alpha}_{T,0})$ constructed in the last proof satisfy the following properties:
\begin{itemize}
	\item[(P1)] for $0 \leq s \leq t$,
	\begin{align}
		\Fcal \mu_{T,0}^{t,\alpha} (\xi) = - \frac{i \xi}{2} \Fcal K^{t,\alpha} (t,\xi),\\
		\Fcal  K^{t,\alpha}_T (s,\xi) = \frac{1}{2} \Big[\frac{e^{s\xi^2/2} - e^{(t-s)\xi^2/2}}{e^{t\xi^2/2}-1} +1\Big] \Fcal K^{t,\alpha} (t,\xi),
	\end{align}
	where
	\[\Fcal K^{t,\alpha}(t,\xi) = \frac{\alpha (1-e^{-t \xi^2/2})}{\sqrt{t} \xi^2};\]
	\item[(P2)] for $t \leq s \leq T$,
	\begin{equation}
		\Fcal  K^{t,\alpha}_T (s,\xi)  = \frac{1}{2} \big[1+e^{-(s-t)\xi^2/2}\big] \Fcal K^{t,\alpha} (t,\xi).
	\end{equation}
\end{itemize}
Moreover, by direct calculations (see Section \ref{appendix:pf_current_fourier}), for $0 \leq s \leq T$,
\begin{equation}\label{current_fourier}
K^{t,\alpha}_T (s,0) =	\frac{1}{\sqrt{2\pi}} \int_{\R} \Fcal  K^{t,\alpha}_T (s,\xi) d\xi= \frac{\alpha a(t,s)}{\sqrt{t}},
\end{equation}
where $a(t,s) = \frac{1}{2} (\sqrt{t} + \sqrt{s} - \sqrt{|t-s|})$ is the covariance function  of the fractional Brownian motion with Hurst parameter $\frac{1}{4}$.

(P1) is exactly \eqref{mu_0_fourier_2} and \eqref{K_fourier_2}. (P2) follows from \eqref{mu_T_t_alpha_2} and \eqref{mu_K}.  Indeed, by \eqref{mu_T_t_alpha_2}, 
\[\Fcal \mu_{T}^{t,\alpha}  (s,\xi) = \Fcal \mu^{t,\alpha} (t,\xi)  e^{-(s-t)\xi^2/2}, \quad t \leq s \leq T,\] 
and by \eqref{mu_K},
\begin{equation}\label{mu_K_fourier}
	- i \xi \Fcal K^{t,\alpha}_T  (s,\xi) = \Fcal \mu_{T,0}^{t,\alpha} (\xi)  - \Fcal \mu_{T}^{t,\alpha}  (s,\xi), \quad s \geq 0.
\end{equation}
Taking $s = t$ in the last expression, and then using (P1), we obtain
\[\Fcal \mu_{T}^{t,\alpha}  (t,\xi) = \frac{i \xi}{2} \Fcal K^{t,\alpha}  (t,\xi).\]
Therefore, by \eqref{mu_K_fourier}, for $s \geq t$,
\[- i \xi \Fcal K^{t,\alpha}_T  (s,\xi) = - \frac{i \xi}{2} \Fcal K^{t,\alpha} (t,\xi) - \Fcal \mu^{t,\alpha} (t,\xi)  e^{-(s-t)\xi^2/2} = - \frac{i \xi}{2} \big[1+e^{-(s-t)\xi^2/2}\big] \Fcal K^{t,\alpha} (t,\xi),\]
as claimed.

Finally, we are ready to prove Proposition \ref{prop:Q_finiteD}.

\begin{proof}[Proof of Proposition \ref{prop:Q_finiteD}]
	Let \[\mu_T^{n} := \mu_T^{n,\{t_j\}_{j=1}^n,\{\alpha_j\}_{j=1}^n}, \quad (K_T^{n}, \mu_{T,0}^{n}) := (K_T^{n,\{t_j\}_{j=1}^n,\{\alpha_j\}_{j=1}^n}, \mu_{T,0}^{n,\{t_j\}_{j=1}^n,\{\alpha_j\}_{j=1}^n})\] be the minimizer of \eqref{Q_n}.  We claim that
	\begin{equation}\label{k_n_mu_n}
		K_T^{n} = \sum_{j=1}^{n} \beta_j K_T^{t_j,1}, \quad 	\mu_{T,0}^{n} = \sum_{j=1}^{n} \beta_j \mu_{T,0}^{t_j,1}
	\end{equation}
	for some coefficients $\beta_j \in \R$, $1 \leq j \leq n$.  Recall that for $1 \leq j \leq n$, $\big(K_T^{t_j,1}, \mu_{T,0}^{t_j,1}\big)$ is the minimizer of the rate function $\Qcal_T (\mu)$ with current one  at time $t_j$, and is constructed in \eqref{mu_T_t_alpha_1} and \eqref{mu_T_t_alpha_2}. We prove claim \eqref{k_n_mu_n} in Section \ref{appendix:k_nmu_n}. The coefficients $\{\beta_j\}$ satisfy the following constraints
	\begin{equation}\label{equ:constraints}
		\sum_{j=1}^{n} \beta_j K_T^{t_j,1} (t_k,0) = \alpha_k, \quad 1 \leq k \leq n.
	\end{equation}
	By \eqref{current_fourier}, the last line is equivalent to
	\[\sum_{j=1}^{n} \beta_j \frac{a(t_j,t_k)}{\sqrt{t_j}}= \alpha_k, \quad 1 \leq k \leq n.\]
	Let  $D = {\rm diag } (d_j)_{1 \leq j \leq n}$ be the diagonal matrix with $d_j = 1/ \sqrt{t_j}$.  Let  $\betabf=(\beta_1,\ldots,\beta_n)^T$ be the column vector. Then, we have shown that
	\[A D \betabf = \alphabf.\]
	Note that $AD$ is invertible, which ensures the existence and uniqueness of $\betabf$.
	
	Inserting \eqref{k_n_mu_n} into \eqref{parseval}, we have 
	\begin{equation}
		\chi (\rho) \mathcal{Q}_T (\mu_{T,0}^n,K^n_T) = \betabf^T Q \betabf,
	\end{equation}
	where $Q=(q_{jk})_{1 \leq j,k \leq n}$ is the matrix with $q_{jk}$ being the real part of
	\begin{multline*}
		\int_0^T \int_{\R} \Big\{ \frac{1}{2}  \partial_{t} \Fcal K_T^{t_j,1} (t,\xi) \overline{\partial_{t} \Fcal K_T^{t_k,1} (t,\xi)}
		\\
		+ \frac{1}{8} \big[- i \xi \Fcal \mu_{T,0}^{t_j,1} (\xi) + \xi^2  \Fcal K_T^{t_j,1} (t,\xi) \big] \overline{\big[- i \xi \Fcal \mu_{T,0}^{t_k,1} (\xi) + \xi^2  \Fcal K_T^{t_k,1} (t,\xi) \big]} \Big\} d\xi\,dt\\
		+ \frac{1}{4} \int_{\R} \Big\{\Fcal \mu_{T,0}^{t_j,1} (\xi) \overline{\Fcal \mu_{T,0}^{t_k,1} (\xi)} \\
		+ \big[\Fcal \mu_{T,0}^{t_j,1}  (\xi)  + i \xi  \Fcal K_{T}^{t_j,1} (T,\xi) \big] \overline{\big[\Fcal \mu_{T,0}^{t_k,1}  (\xi)  + i \xi  \Fcal K_{T}^{t_k,1} (T,\xi) \big]}\Big\} d \xi.
	\end{multline*}
	After a long computation (see Section \ref{appendix:pf_q_jk}), we have
	\begin{equation}\label{q_jk}
		q_{jk} = \frac{\sqrt{2 \pi}}{4 \sqrt{t_jt_k}} a_{jk}.
	\end{equation}
	Equivalently,
	\[Q = \frac{\sqrt{2\pi}}{4} D A D.\]
	Therefore, 
	\[	\chi (\rho) \mathcal{Q}_T (\mu_{T,0}^n,K^n_T) =  \frac{\sqrt{2\pi}}{4} \alphabf^TA^{-1} \alphabf,\]
	which concludes the proof.
\end{proof}

\section{Exponential tightness}\label{sec:exp tight}

In this section, we prove the sequences of the two processes \[\{J_{-1,0}(tN^2)/a_N: 0 \leq t \leq T\}_{N \geq 1} \quad  \text{and} \quad  \{X (tN^2)/a_N: 0 \leq t \leq T\}_{N \geq 1}\] are exponentially tight.

\subsection{A uniform super-exponential estimate} We first state a super-exponential estimate for the space average of the current. The following result was proved by the same authors in \cite{xue2023moderate} without the supremum over time inside the probability.

\begin{lemma}\label{lem:superexponential_estimate}
	For any $\delta > 0$,
	\begin{equation}\label{uniform_see}
		\limsup_{n \rightarrow \infty}\,\limsup_{N \rightarrow \infty} \frac{N}{a_N^2} \log \P \Big(\sup_{0 \leq t \leq T}\Big|\frac{1}{nNa_N} \sum_{x=0}^{nN-1} J_{x,x+1} (tN^2)\Big| > \delta \Big) = - \infty
	\end{equation}
under both $\P=\P_{\nu_\rho}$ and $\P=\P_{\nu_\rho^*}$.
\end{lemma}

\begin{proof}
By Remark \ref{rmk:coupling}, we only need to discuss the case where $\P=\P_{\nu_\rho}$.  By Markov's inequality, for any $K >0$, the expression in \eqref{uniform_see} is bounded by
	\begin{equation}
	-\delta K + \frac{N}{a_N^2} \log \E_{\nu_\rho} \Big[\exp \Big\{ \sup_{0 \leq t \leq T}\Big|\frac{a_N K}{nN^2} \sum_{x=0}^{nN-1} J_{x,x+1} (tN^2)\Big| \Big\}\Big].
	\end{equation}
Since $K$ could be taken arbitrarily large, we only need to prove that, for any fixed $K > 0$,
\[	\limsup_{n \rightarrow \infty}\,\limsup_{N \rightarrow \infty} \frac{N}{a_N^2} \log \E_{\nu_\rho} \Big[\exp \Big\{ \sup_{0 \leq t \leq T}\Big|\frac{a_N K}{nN^2} \sum_{x=0}^{nN-1} J_{x,x+1} (tN^2)\Big| \Big\}\Big] = 0.\]
By using the inequality \[\exp\{\sup_{0 \leq t \leq T} |x_t| \}\leq \exp\{\sup_{0 \leq t \leq T} x_t \}+\exp\{ \sup_{0 \leq t \leq T} (-x_t) \}\] for any trajectory $\{x_t\}$, and $\log (a+b) \leq \log 2 + \max\{\log a, \log b\}$ for any $a,b > 0$,  without loss of generality, it suffices to prove that, for any fixed $K >0$, 
\begin{equation}\label{uniform_see1}
\limsup_{n \rightarrow \infty}\,\limsup_{N \rightarrow \infty} \frac{N}{a_N^2} \log \E_{\nu_\rho} \Big[\exp \Big\{ \sup_{0 \leq t \leq T} \frac{a_N K}{nN^2} \sum_{x=0}^{nN-1} J_{x,x+1} (tN^2)  \Big\}\Big] = 0.
\end{equation}

Since for each $x \in \Z$, $\{J_{x,x+1} (t)\}$ is a compound Poisson process with intensity \[\frac{1}{2}\int_0^t (\eta_s(x) - \eta_s (x+1)) ds,\] and there are no common jumps between those compound Poisson processes, 
\[M^{N,1}_t = \exp \Big\{ \frac{ a_N K}{nN^2} \sum_{x=0}^{nN-1} J_{x,x+1} (tN^2)  -  \Gamma^{N,1}_{n,K} (t)\Big\}\]
is a mean-one exponential martingale, where
\[\Gamma^{N,1}_{n,K} (t) = \frac{1}{2} \sum_{x=0}^{nN-1} \Big(e^{a_N K/(nN^2)} - 1\Big) \int_0^{tN^2} (\eta_s (x) - \eta_s (x+1)) ds.\]
Then, by Cauchy-Schwarz inequality, the expectation in \eqref{uniform_see1} is bounded by
\begin{multline*}
\E_{\nu_\rho} \Big[\sup_{0 \leq t \leq T} \{ M_t^{N,1}\} \exp \Big\{ \sup_{0 \leq t \leq T}  \Gamma^{N,1}_{n,K} (t)  \Big\}\Big] \\
\leq \E_{\nu_\rho} \Big[\sup_{0 \leq t \leq T}  (M_t^{N,1})^2\Big]^{1/2} \E_{\nu_\rho} \Big[ \exp \Big\{ \sup_{0 \leq t \leq T}  2 \Gamma^{N,1}_{n,K} (t)  \Big\}\Big]^{1/2}.
\end{multline*}
As a consequence, the expression in \eqref{uniform_see1} is bounded by
\[\frac{N}{2a_N^2} \log \E_{\nu_\rho} \Big[\sup_{0 \leq t \leq T}  (M_t^{N,1})^2\Big] + \frac{N}{2a_N^2} \log  \E_{\nu_\rho} \Big[ \exp \Big\{ \sup_{0 \leq t \leq T}  2 \Gamma^{N,1}_{n,K} (t)  \Big\}\Big].\]

For the martingale term in the last line, we have
\begin{multline*}
	(M_t^{N,1})^2 = \exp \Big\{ \frac{1}{2}\Big[ \frac{ 4 a_N K}{nN^2} \sum_{x=0}^{nN-1} J_{x,x+1} (tN^2)  - 2 \Gamma^{N,2}_{n,K} (t) \Big]+ \Gamma^{N,2}_{n,K} (t)  -  2 \Gamma^{N,1}_{n,K} (t)\Big\}\\
	=: 	(M_t^{N,2})^{1/2} \exp \Big\{  \Gamma^{N,2}_{n,K} (t)  -  2 \Gamma^{N,1}_{n,K} (t)\Big\},
\end{multline*}
where
\[2 \Gamma^{N,2}_{n,K} (t) = \frac{1}{2} \sum_{x=0}^{nN-1} \Big(e^{4 a_N K/(nN^2)} - 1\Big) \int_0^{tN^2} (\eta_s (x) - \eta_s (x+1)) ds,\]
and 
\[M_t^{N,2} = \exp \Big\{ \frac{ 4 a_N K}{nN^2} \sum_{x=0}^{nN-1} J_{x,x+1} (tN^2)  - 2 \Gamma^{N,2}_{n,K} (t) \Big\} \]
is also a martingale. By Doob's inequality and Cauchy-Schwarz inequality,
\begin{multline}\label{uniform_see2}
\frac{N}{a_N^2} \log \E_{\nu_\rho} \Big[\sup_{0 \leq t \leq T}  (M_t^{N,1})^2\Big]  \leq \frac{N \log 4}{ a_N^2} + \frac{N}{a_N^2} \log \E_{\nu_\rho} \Big[ (M_T^{N,1})^2\Big]\\
\leq \frac{N \log 4}{a_N^2} +  \frac{N}{2 a_N^2} \log \E_{\nu_\rho} \Big[ M_T^{N,2}\Big] + \frac{N}{2 a_N^2} \log \E_{\nu_\rho} \Big[\exp \Big\{ 2 \Gamma^{N,2}_{n,K} (T)  -  4 \Gamma^{N,1}_{n,K} (T)\Big\} \Big].
\end{multline}
Note that the second term on the right-hand side of \eqref{uniform_see2} is zero. Since $a_N \gg \sqrt{N}$, the first term on the right-hand side of \eqref{uniform_see2} converges to zero as $N \rightarrow \infty$.  Therefore, to prove \eqref{uniform_see1}, we only need to prove
\begin{equation}
\limsup_{n \rightarrow \infty}\,\limsup_{N \rightarrow \infty}  \frac{N}{2a_N^2} \log  \E_{\nu_\rho} \Big[ \exp \Big\{ \sup_{0 \leq t \leq T}  2 \Gamma^{N,1}_{n,K} (t)  \Big\}\Big] = 0,
\end{equation}
since the third term on the right-hand side of \eqref{uniform_see2} could be treated in the same way.

Since $|e^x - 1 -x | \leq (x^2/2) e^{|x|}$, 
\[ \Gamma^{N,1}_{n,K} (t) \leq \frac{a_N K}{2 nN^2} \int_0^{tN^2} (\eta_s (0) - \eta_s (nN)) ds +  \frac{C_{N,n} ta_N^2 K^2}{n^2 N^2},\]
where $C_{N,n}$ is uniformly bounded in $N$ and $n$. Thus, we are left to show 
\begin{equation}\label{uniform_see3}
	\limsup_{n \rightarrow \infty}\,\limsup_{N \rightarrow \infty}  \frac{N}{a_N^2} \log  \E_{\nu_\rho} \Big[ \exp \Big\{ \sup_{0 \leq t \leq T} \frac{ a_N K}{nN^2} \int_0^{tN^2} (\eta_s (0) - \eta_s (nN)) ds  \Big\}\Big] = 0.
\end{equation}
Since 
\[\sup_{0 \leq t \leq T} \frac{ a_N K}{nN^2} \int_0^{tN^2} (\eta_s (0) - \eta_s (nN)) ds \leq \sup_{0 \leq k \leq TN} \frac{ a_N K}{nN^2} \int_0^{kN} (\eta_s (0) - \eta_s (nN)) ds + \frac{a_NK}{nN},\]
where the supremum on the right-hand side is over integers $k$, the left-hand side in \eqref{uniform_see3} is bounded by
\[\limsup_{n \rightarrow \infty}\,\limsup_{N \rightarrow \infty}  \frac{N}{a_N^2} \log  \E_{\nu_\rho} \Big[ \exp \Big\{ \sup_{0 \leq k \leq TN} \frac{a_N K}{nN^2} \int_0^{kN} (\eta_s (0) - \eta_s (nN)) ds  \Big\}\Big].\]
Using the inequality
\[\log \E [\exp \{\sup_{0 \leq k \leq TN} X_k\}] \leq \log \E \Big[\sum_{k=0}^{TN} e^{X_k} \Big] \leq \log (TN+1)  + \max_{0 \leq k \leq TN} \log \E [e^{X_k}],\]
and since $a_N \gg \sqrt{N \log N}$, we only need to prove
\begin{equation}\label{uniform_see4}
\limsup_{n \rightarrow \infty}\,\limsup_{N \rightarrow \infty}  \max_{0 \leq k \leq TN} \frac{N}{a_N^2} \log  \E_{\nu_\rho} \Big[ \exp \Big\{  \frac{a_N K}{nN^2} \int_0^{kN} (\eta_s (0) - \eta_s (nN)) ds  \Big\}\Big] = 0,
\end{equation}
which has been shown in the proof of \cite[Lemma 4.1]{xue2023moderate}.  We also refer the readers to Section \ref{app_pf_uniform_see4} for details. This concludes the proof.
\end{proof}

\subsection{Exponential tightness of the current}  To show the sequence of processes $\{J_{-1,0}(tN^2)/a_N: 0 \leq t \leq T\}_{N \geq 1}$ is exponentially tight, we only need to prove the following estimates for the current.

\begin{lemma}\label{lem:current exponential tight}
We have the following estimates for the current under both $\P=\P_{\nu_\rho}$ and $\P=\P_{\nu_\rho^*}$:
\begin{enumerate}
	\item for any fixed $T > 0$,
	\begin{equation}\label{et1}
		\limsup_{M \rightarrow \infty} \limsup_{N \rightarrow \infty} \frac{N}{a_N^2} \log \P \Big(\sup_{0 \leq t \leq T}  \frac{1}{a_N} \big|  J_{-1,0} (tN^2)\big| > M\Big) = - \infty;
	\end{equation}
\item for any $\varepsilon > 0$,
\begin{equation}\label{et2}
	\limsup_{\delta \rightarrow 0} \limsup_{N \rightarrow \infty}  \sup_{\tau \in \mathcal{T}_T} \frac{N}{a_N^2} \log \P \Big(\sup_{0< t \leq \delta} \frac{1}{a_N} \big|  J_{-1,0} ((t+\tau)N^2) - J_{-1,0} (\tau N^2)\big| > \varepsilon\Big) = - \infty,
\end{equation}
where $\mathcal{T}_T$ is the set of all stopping times bounded by $T$.
\end{enumerate}
\end{lemma}

\begin{proof}
As in the last lemma, we  only need to deal with the case where $\P=\P_{\nu_\rho}$.
For any $n > 0$, introduce the function $G_n: \R \rightarrow \R$ as
\[G_n (u) = (1-u/n)^+ \mathbf{1} \{ u \geq 0\}, \quad u \in \R. \]
Then, one could check directly that, for any $n > 0$,
 \begin{equation}\label{et2.2}
 \frac{1}{a_N} J_{-1,0} (tN^2) = \<\mu^N_t,G_n\> - \<\mu^N_0,G_n\>  + \frac{1}{nNa_N} \sum_{x=0}^{nN-1} J_{x,x+1} (tN^2).
 \end{equation}
Thus, we only need to show that the estimates in the lemma hold respectively for the terms on the right-hand side of the last line, \emph{i.e.}, 
\begin{align}
	&\limsup_{n \rightarrow \infty}	\limsup_{M \rightarrow \infty} \limsup_{N \rightarrow \infty} \frac{N}{a_N^2} \log \P_{\nu_\rho} \Big(\sup_{0 \leq t \leq T} \Big|  \frac{1}{nNa_N} \sum_{x=0}^{nN-1} J_{x,x+1} (tN^2)   \Big| > M\Big) = - \infty,\label{et4}\\
		&\limsup_{n \rightarrow \infty}		\limsup_{\delta \rightarrow 0} \limsup_{N \rightarrow \infty}  \sup_{\tau \in \mathcal{T}_T} \notag\\
	&\quad \frac{N}{a_N^2} \log \P_{\nu_\rho} \Big(\sup_{0 \leq t \leq \delta} \Big|  \frac{1}{nNa_N} \sum_{x=0}^{nN-1} [J_{x,x+1} ((t+\tau)N^2)  - J_{x,x+1} (\tau N^2)]  \Big| > \varepsilon \Big) = - \infty,\label{et6}
\end{align}
and
\begin{align}
\limsup_{n \rightarrow \infty}	\limsup_{M \rightarrow \infty} \limsup_{N \rightarrow \infty} \frac{N}{a_N^2} \log \P_{\nu_\rho} \Big(\sup_{0 \leq t \leq T} \big| \<\mu^N_t,G_n\>   \big| > M\Big) &= - \infty,\label{et3}\\
\limsup_{n \rightarrow \infty} \limsup_{\delta \rightarrow 0} \limsup_{N \rightarrow \infty}  \sup_{\tau \in \mathcal{T}_T} \frac{N}{a_N^2} \log \P_{\nu_\rho} \Big(\sup_{0 \leq t \leq \delta} \big| \<\mu^N_{t+\tau} - \mu^N_\tau,G_n\>   \big| > \varepsilon \Big) &= - \infty. \label{et5}
\end{align}

Recall that \eqref{et4} has been proved in Lemma \ref{lem:superexponential_estimate}.  Since the probability in \eqref{et6} is bounded by 
\[ 2 \P_{\nu_\rho} \Big(\sup_{0 \leq t \leq T} \Big|  \frac{1}{nNa_N} \sum_{x=0}^{nN-1} J_{x,x+1} (tN^2)   \Big| > \varepsilon/2\Big),\]
the estimate in \eqref{et6} also follows directly from Lemma \ref{lem:superexponential_estimate}.

The estimates in  \eqref{et3}  and \eqref{et5} have been proved in \cite[Lemma 3.2]{gao2003moderate} for Schwartz test functions. Thus, we only need to find $\tilde{G}_n \in \Scal (\R)$ such that, for any $\varepsilon > 0$, 
\begin{equation}\label{smooth out G_n}
	\limsup_{n \rightarrow \infty}	\limsup_{N \rightarrow \infty} \frac{N}{a_N^2} \log \P_{\nu_\rho} \Big(\sup_{0 \leq t \leq T} \big| \<\mu^N_t,G_n - \widetilde{G}_n\>   \big| > \varepsilon\Big) = - \infty.
\end{equation}
We underline that this is not obvious since we only have the priori bound 
\[\big| \<\mu^N_t,G_n - \widetilde{G}_n\>   \big| \leq C(n) \frac{N}{a_N}\]
and note that $a_N \ll N$. To this end, we choose $\tilde{G}_n \in \Scal (\R)$ such that 
\[|\delta_n (u)| := |G_n (u) - \tilde{G}_n (u)| \leq C n^{-1}, \quad \forall u \in \R,\]
and that the support of $\delta_n (\cdot)$ is contained in $[-2n,2n]$. For $0 \leq i \leq N^3$, let $t_i = i T / N^{3}$. Then, the probability in \eqref{smooth out G_n} is bounded by
\begin{equation}\label{et13}
	\P_{\nu_\rho} \Big(\sup_{0 \leq i \leq N^3} \big| \<\mu^N_{t_i},\delta_n\>   \big| > \varepsilon/2\Big) 
	+ \P_{\nu_\rho} \Big(\sup_{0 \leq t \leq T} \big| \<\mu^N_t,\delta_n\>   \big|  - \sup_{0 \leq i \leq N^3} \big| \<\mu^N_{t_i},\delta_n\>   \big| > \varepsilon/2\Big).
\end{equation}

For the first term in the last line, by stationary of the process and since $a_N \gg \sqrt{N \log N}$,
\begin{multline*}
	\limsup_{n \rightarrow \infty}	\limsup_{N \rightarrow \infty} \frac{N}{a_N^2} \log 	\P_{\nu_\rho} \Big(\sup_{0 \leq i \leq N^3} \big| \<\mu^N_{t_i},\delta_n\>   \big| > \varepsilon/2\Big)\\
	 \leq 	\limsup_{n \rightarrow \infty}	\limsup_{N \rightarrow \infty} \frac{N}{a_N^2} \log 	\P_{\nu_\rho} \Big( \big| \<\mu^N_{0},\delta_n\>   \big| > \varepsilon/2\Big) \\
	\leq - \frac{A\varepsilon}{2} + \limsup_{n \rightarrow \infty}	\limsup_{N \rightarrow \infty} \frac{N}{a_N^2} \log 	\E_{\nu_\rho} \Big[ \exp \Big\{  \frac{a_N A}{N} \big| \sum_{x \in \Z} \bar{\eta}_0 (x) \delta_n (x/N) \big| \Big\}  \Big].
\end{multline*}
As before, we could remove the absolute value inside the exponential in the last line. Since $\nu_\rho$ is a product measure, and by Taylor's expansion up to second order, we bound the last line by
\begin{multline*}
- \frac{A\varepsilon}{2}  +  \limsup_{n \rightarrow \infty}	\limsup_{N \rightarrow \infty} \frac{N}{a_N^2} \sum_{x \in \Z} \log \E_{\nu_\rho} \Big[ \exp \Big\{  \frac{a_N A}{N} \bar{\eta}_0 (x) \delta_n (x/N) \Big\}  \Big]\\
\leq - \frac{A\varepsilon}{2}  +  \limsup_{n \rightarrow \infty}	\limsup_{N \rightarrow \infty} \frac{C A^2}{N}  \sum_{x \in \Z} \delta_n^2 (x/N)  = - \frac{A\varepsilon}{2}.
\end{multline*}
Since $A$ could be arbitrarily large, 
\[	\limsup_{n \rightarrow \infty}	\limsup_{N \rightarrow \infty} \frac{N}{a_N^2} \log 	\P_{\nu_\rho} \Big(\sup_{0 \leq i \leq N^3} \big| \<\mu^N_{t_i},\delta_n\>   \big| > \varepsilon/2\Big) = - \infty.\]

It remains to bound the second term in \eqref{et13}. Before that, we first recall the stirring representation for the symmetric exclusion process. Initially, we put a particle at each site of $\Z$.  Those particles are distinguishable.  For each bond $(x,x+1)$, we have a Poisson process with parameter $\tfrac{1}{2}$, and assume those Poisson processes are independent.  When the Poisson clock associated with the bond $(x,x+1)$ rings, we  exchange the particles at sites $x$ and $x+1$.  Let $\xi_t^x$ be the position at time $t$ of the particle initially at site $x$. For $\eta \in \Omega$, let
\[\eta_t (x) = 1 \quad \text{if and only if} \quad  \text{$\xi_t^y = x$ for some $y \in \Z$ such that $\eta(y) = 1$}.\]
Then, $\{\eta_t\}$ is a version of the SSEP with initial configuration $\eta$.  Note that for each $x$, $\{\xi_t^x\}$ is a continuous-time simple random walk with jump rate one. Using this representation, for any $s,t \geq  0$, we have
\[\<\mu^N_t - \mu^N_s,\delta_n\> = \frac{1}{a_N} \sum_{y \in \Z} \Big[\delta_n \Big(\frac{\xi^y_{tN^2}}{N}\Big)  - \delta_n \Big(\frac{\xi^y_{sN^2}}{N}\Big)\Big] \eta_0 (y).\]
Thus, we bound the second term in \eqref{et13} by 
\begin{multline*}
 \P_{\nu_\rho} \Big(\sup_{0 \leq t \leq T} \big| \<\mu^N_t,\delta_n\>   \big|  - \sup_{0 \leq i \leq N^3} \big| \<\mu^N_{t_i},\delta_n\>   \big| > \varepsilon/2\Big) \\
	\leq   \P_{\nu_\rho} \Big(\sup_{0 \leq i \leq N^3} \sup_{t_i \leq t \leq t_{i+1}} \big| \<\mu^N_t - \mu^N_{t_i},\delta_n\>   \big|   > \varepsilon/2\Big) \leq N^3  \P_{\nu_\rho} \Big( \sup_{0 \leq t \leq t_{1}} \big| \<\mu^N_t - \mu^N_{0},\delta_n\>   \big|   > \varepsilon/2\Big) \\
	= N^3  \P_{\nu_\rho} \Big( \sup_{0 \leq t \leq t_{1}} \Big| \frac{1}{a_N} \sum_{y \in \Z} \Big[\delta_n \Big(\frac{\xi^y_{tN^2}}{N}\Big)  - \delta_n \Big(\frac{y}{N}\Big)\Big] \eta_0 (y)  \Big|   > \varepsilon/2\Big).
\end{multline*}
To conclude the proof, it is sufficient to show that, for any $n > 0$, 
\begin{equation}\label{et14}
\limsup_{N \rightarrow \infty} \frac{N}{a_N^2} \log  \P_{\nu_\rho} \Big( \sup_{0 \leq t \leq t_{1}} \Big| \frac{1}{a_N} \sum_{y \in \Z} \Big[\delta_n \Big(\frac{\xi^y_{tN^2}}{N}\Big)  - \delta_n \Big(\frac{y}{N}\Big)\Big] \eta_0 (y)  \Big|   > \varepsilon/2\Big) = -\infty.
\end{equation}
For the above sum over $y \in \Z$, we consider the two cases  $|y| >  3 N n$ and $|y| \leq  3 Nn$ respectively. 

For the case $|y| > 3 N n$, $\delta_n (y/N) = 0$ since the support of $\delta_n$ is contained in $[-2n,2n]$. Then, we  bound 
\begin{multline*}
	 \P_{\nu_\rho} \Big( \sup_{0 \leq t \leq t_{1}} \Big| \frac{1}{a_N} \sum_{|y| > 3 Nn} \delta_n \Big(\frac{\xi^y_{tN^2}}{N}\Big)  \eta_0 (y)  \Big|   > \varepsilon/2\Big) 
	  \leq \P_{\nu_\rho} \Big( \sup_{0 \leq t \leq t_{1}} |\xi^y_{tN^2}| \leq 2 Nn \quad \text{for some $|y| > 3 Nn$}\Big)\\
	 \leq   \sum_{|y| > 3Nn}  \P_{\nu_\rho} \Big( \sup_{0 \leq t \leq t_{1}} |\xi^y_{tN^2} - y| \geq |y|- 2 Nn \Big).
\end{multline*}
Since for each $y$, $|\xi^y_{tN^2} - y|$ is stochastically bounded by a Poisson process with parameter $tN^2$, and recall $t_1 = T/N^3$, the last line is bound by 
\[ \sum_{|y| > 3Nn}  e^{-c(|y|-2Nn)} \leq e^{-cNn}\]
for some constant $c > 0$.  Therefore,
\[\limsup_{N \rightarrow \infty} \frac{N}{a_N^2} \log   \P_{\nu_\rho} \Big( \sup_{0 \leq t \leq t_{1}} \Big| \frac{1}{a_N} \sum_{|y| > 3 Nn} \delta_n \Big(\frac{\xi^y_{tN^2}}{N}\Big)  \eta_0 (y)  \Big|   > \varepsilon/2\Big) = - \infty.\]

It remains to bound
\begin{multline}\label{et15}
 \P_{\nu_\rho} \Big( \sup_{0 \leq t \leq t_{1}} \Big| \frac{1}{a_N} \sum_{|y| \leq 3 Nn} \Big[\delta_n \Big(\frac{\xi^y_{tN^2}}{N}\Big)  - \delta_n \Big(\frac{y}{N}\Big)\Big] \eta_0 (y)  \Big|   > \varepsilon/2\Big) \\
 \leq \P_{\nu_\rho} (A_N^c) +  \P_{\nu_\rho} \Big( \Big\{\sup_{0 \leq t \leq t_{1}} \Big| \frac{1}{a_N} \sum_{|y| \leq 3 Nn} \Big[\delta_n \Big(\frac{\xi^y_{tN^2}}{N}\Big)  - \delta_n \Big(\frac{y}{N}\Big)\Big] \eta_0 (y)  \Big|   > \varepsilon/2 \Big\} \cap A_N \Big),
\end{multline}
where
\[A_N = \Big\{  \sup_{|y| \leq 3 Nn} \sup_{0 \leq t \leq t_1} |\xi^y_{tN^2} - y| \leq a_N^{3/2} N^{-1/2}\Big\}.\]
Standard large deviation estimates yield
\[\P_{\nu_\rho} (A_N^c) \leq 7 Nn e^{-c a_N^{3/2} N^{-1/2}}\]
for some constant $c > 0$, and thus
\[\limsup_{N \rightarrow \infty} \frac{N}{a_N^2} \log  \P_{\nu_\rho} (A_N^c) = - \infty.\]
We claim that the event inside the second probability on the right-hand side of \eqref{et15} is  empty for $N$ large enough, which is sufficient to conclude the proof.  Indeed, first note that $\delta_n$ is only discontinuous at the origin. Moreover, on the event $A_N$, for each $0 \leq t \leq t_1$, the cardinality of \[B_{N,t} := \{|y| \leq 3Nn: y \geq 0, \xi^y_{tN^2} < 0 \;  \text{or}  \; y < 0, \xi^y_{tN^2} \geq 0\}\] is bounded by $2 a_N^{3/2} N^{-1/2}$.  Thus, on the event $A_N$,
\[\sup_{0 \leq t \leq t_{1}} \Big| \frac{1}{a_N} \sum_{y \in B_{N,t}} \Big[\delta_n \Big(\frac{\xi^y_{tN^2}}{N}\Big)  - \delta_n \Big(\frac{y}{N}\Big)\Big] \eta_0 (y)  \Big|  \leq \frac{4}{n} \sqrt{\frac{a_N}{N}},\]
which goes to zero as $N \rightarrow \infty$. By the piecewise smoothness of $\delta_n$, 
\[\sup_{0 \leq t \leq t_{1}} \Big| \frac{1}{a_N} \sum_{|y| \leq 3Nn, y \notin B_{N,t}} \Big[\delta_n \Big(\frac{\xi^y_{tN^2}}{N}\Big)  - \delta_n \Big(\frac{y}{N}\Big)\Big] \eta_0 (y)  \Big|  \leq C(n) \sqrt{\frac{a_N}{N}},\]
which also vanishes as $N \rightarrow \infty$, thus concluding the proof.
\end{proof}

\subsection{Exponential tightness of the tagged particle} As in the last subsection, if follows from the following estimates that the sequence of processes $\{X (tN^2)/a_N: 0 \leq t \leq T\}_{N \geq 1}$ is exponentially tight.

\begin{lemma}\label{lem:tagged particle exponential tight}
	We have the following estimates for the replacement of the tagged particle:
	\begin{enumerate}
		\item for any fixed $T > 0$,
		\begin{equation}\label{et7}
			\limsup_{M \rightarrow \infty} \limsup_{N \rightarrow \infty} \frac{N}{a_N^2} \log \P_{\nu_\rho^*} \Big(\sup_{0 \leq t \leq T}  \frac{1}{a_N} \big|  X (tN^2)\big| > M\Big) = - \infty;
		\end{equation}
		\item for any $\varepsilon > 0$,
		\begin{equation}\label{et8}
			\limsup_{\delta \rightarrow 0} \limsup_{N \rightarrow \infty}  \sup_{\tau \in \mathcal{T}_T} \frac{N}{a_N^2} \log \P_{\nu_\rho^*} \Big(\sup_{0< t \leq \delta} \frac{1}{a_N} \big|  X ((t+\tau)N^2) - X(\tau N^2)\big| > \varepsilon\Big) = - \infty,
		\end{equation}
		where $\mathcal{T}_T$ is the set of all stopping times bounded by $T$.
	\end{enumerate}
\end{lemma}

\begin{proof}
We first prove \eqref{et7}. Since the process is symmetric with respect to the origin, we could remove the absolute value inside the probability in \eqref{et7}. If $X(tN^2) > a_N M$ for some $0 \leq t \leq T$, then $J_{-1,0} (tN^2) \geq \sum_{x=0}^{[a_N M]} \eta_{t N^2} (x)$ for some $0 \leq t \leq T$. Thus, the probability in \eqref{et7} is bounded by
\begin{multline*}
	\P_{\nu_\rho^*} \Big(\sup_{0 \leq t \leq T} J_{-1,0} (tN^2)  \geq \inf_{0 \leq t \leq T} \sum_{x=0}^{[a_N M]} \eta_{tN^2} (x) \Big) \leq 	\P_{\nu_\rho^*} \Big(\sup_{0 \leq t \leq T} J_{-1,0} (tN^2)  \geq \rho a_N M/2 \Big) \\+ 	\P_{\nu_\rho^*} \Big( \inf_{0 \leq t \leq T} \sum_{x=0}^{[a_N M]} \eta_{t N^2} (x) \leq \rho a_N M/2\Big).
\end{multline*}
To finish the proof, it remains to show that 
\begin{equation*}
	\limsup_{M \rightarrow \infty} \limsup_{N \rightarrow \infty} \frac{N}{a_N^2} \log  \P_{\nu_\rho^*} \Big(\sup_{0 \leq t \leq T} J_{-1,0} (tN^2)  \geq \rho a_N M/2 \Big) = - \infty,
\end{equation*}
which follows directly from \eqref{et1}, and that, for any fixed $M > 0$, 
\begin{equation}\label{et9}
\limsup_{N \rightarrow \infty} \frac{N}{a_N^2} \log \P_{\nu_\rho^*} \Big( \inf_{0 \leq t \leq T} \sum_{x=0}^{[a_N M]} \eta_{tN^2} (x) \leq \rho a_N M/2\Big) = - \infty.
\end{equation}

Now, we prove \eqref{et9}. For $0 \leq i \leq N^2$, define $t_i = iT/N^2$.  Then, for $t_i \leq  t \leq t_{i+1}$, 
\[\sum_{x=0}^{[a_N M]} \eta_{t N^2} (x) = \sum_{x=0}^{[a_N M]} \eta_{t_iN^2} (x)+ J_{-1,0} (t_iN^2,tN^2) -  J_{[a_NM],[a_NM]+1} (t_iN^2,tN^2),\]
which implies that 
\begin{multline*}
\inf_{0 \leq t \leq T}	\sum_{x=0}^{[a_N M]} \eta_{t N^2} (x) \geq \inf_{0 \leq i \leq N^2} \sum_{x=0}^{[a_N M]} \eta_{t_iN^2} (x) \\
- \sup_{0 \leq i \leq N^2} \sup_{t_i \leq t \leq t_{i+1}} \big|J_{-1,0} (t_iN^2,tN^2)\big|  - \sup_{0 \leq i \leq N^2} \sup_{t_i \leq t \leq t_{i+1}} \big|J_{[a_NM],[a_NM]+1} (t_iN^2,tN^2)\big|.
\end{multline*}
This permits us to bound the probability in \eqref{et9} by
\begin{multline*}
	\P_{\nu_\rho^*} \Big( \inf_{0 \leq i \leq N^2} \sum_{x=0}^{[a_N M]} \eta_{t_iN^2} (x) \leq 3 \rho a_N M/4 \Big) 
	+ 	\P_{\nu_\rho^*} \Big(  \sup_{0 \leq i \leq N^2} \sup_{t_i \leq t \leq t_{i+1}} \big|J_{-1,0} (t_iN^2,tN^2)\big|  \geq \rho a_N M/8 \Big) \\
	+ 	\P_{\nu_\rho^*} \Big( \sup_{0 \leq i \leq N^2} \sup_{t_i \leq t \leq t_{i+1}} \big|J_{[a_NM],[a_NM]+1} (t_iN^2,tN^2)\big| \geq \rho a_N M/8 \Big).
\end{multline*}
By Remark \ref{rmk:coupling}, we could replace the above $\P_{\nu_\rho^*}$ with $\P_{\nu_\rho}$. Since $\nu_\rho$ is  invariant for the SSEP, using the large deviation principle of the sum of i.i.d. random variables, there exists $C>0$ independent of $N$ and $1\leq i\leq N^2$ such that
\[
\limsup_{N\rightarrow+\infty}\frac{1}{a_N}\log \P_{\nu_\rho}\Big( \sum_{x=0}^{[a_N M]} \eta_{t_iN^2} (x) \leq 3 \rho a_N M/4\Big)=-C
\]
for all $1\leq i\leq N^2$. Then, since $a_N \gg \sqrt{N\log N}$, we have
\begin{equation}\label{et9.1}
\limsup_{N\rightarrow+\infty}\frac{N}{a_N^2}\log \P_{\nu_\rho} \Big(\inf_{0 \leq i \leq N^2} \sum_{x=0}^{[a_N M]} \eta_{t_i N^2} (x) \leq 3 \rho a_N M/4\Big)=-\infty.
\end{equation}
Since a particle crosses an edge at rate at most $1$, $\{\big|J_{-1,0} (t_iN^2,tN^2)\big| \}_{t\geq t_i}$ is stochastically dominated from above by a Poisson process with rate $1$. Thus,
\begin{multline*}
\P_{\nu_{\rho}} (\sup_{0 \leq i \leq N^2} \sup_{t_i \leq t \leq t_{i+1}} \big|J_{-1,0} (t_iN^2,tN^2)\big|  \geq \rho a_N M/8 )\\
 \leq  N^2 \P_{\nu_{\rho}} (\sup_{t_i \leq t \leq t_{i+1}} \big|J_{-1,0} (t_iN^2,tN^2)\big|  \geq \rho a_N M/8 )  \leq C N^2  \exp\{-C a_N M\}
\end{multline*}
for some constant $C > 0$,  which implies
\begin{equation}\label{et9.2}
\limsup_{N\rightarrow+\infty}\frac{N}{a^2_N}\log \P_{\nu_{\rho}} \Big(\sup_{0 \leq i \leq N^2} \sup_{t_i \leq t \leq t_{i+1}} \big|J_{-1,0} (t_iN^2,tN^2)\big|  \geq \rho a_N M/8 \Big)=-\infty.
\end{equation}
By translation invariance, \eqref{et9.2} still holds when $J_{-1,0} (t_iN^2,tN^2)$ is replaced by 
\[J_{[a_NM],[a_NM]+1} (t_iN^2,tN^2),\]
thus concluding the proof of  \eqref{et9}. 

Now we prove \eqref{et8}. We first remove the absolute value inside the probability in \eqref{et8} as before. For any $M>0$, we define \[B_{M,N,\delta}= \Big\{\sup_{0\leq t\leq {T+\delta}}\frac{1}{a_N}|X(tN^2)|\leq M\Big\}.\] By \eqref{et7},
\[
\limsup_{M\rightarrow+\infty}\limsup_{N\rightarrow+\infty}\frac{N}{a_N^2}\log \P_{\nu_{\rho}^*}(B_{M,N,\delta}^c)=-\infty.
\]
Hence, to prove \eqref{et8}, we only need to show that
\begin{equation}\label{et10}
\limsup_{\delta \rightarrow 0} \limsup_{N \rightarrow \infty}  \sup_{\tau \in \mathcal{T}_T} \frac{N}{a_N^2} \log \P_{\nu_{\rho}^*} \Big(\sup_{0< t \leq \delta} \frac{1}{a_N} \left(X ((t+\tau)N^2) - X(\tau N^2)\right) > \varepsilon, B_{N,M,\delta}\Big) = - \infty
\end{equation}
for any $M>0$. We further define 
\[
C_{N,M,\varepsilon,\delta} = \Big\{ \inf_{-a_NM \leq k\leq a_NM} \inf_{0\leq t\leq T+\delta}\sum_{x=0}^{[a_N\varepsilon]}\eta_{tN^2} (x+k)\geq \frac{\rho a_N\varepsilon}{2} \Big\},
\]
where the infimum is over integers $k$. Then, by \eqref{et9}, 
\[
\limsup_{N\rightarrow+\infty}\frac{N}{a_N^2}\log\P_{\nu_{\rho}^*}(C_{N,M,\varepsilon,\delta}^c)=-\infty.
\]
Therefore, to prove \eqref{et10}, we only need to show that, for any $M > 0$,
\begin{multline}\label{et11}
\limsup_{\delta \rightarrow 0} \limsup_{N \rightarrow \infty}  \sup_{\tau \in \mathcal{T}_T} \frac{N}{a_N^2} \log \P_{\nu_{\rho}^*} \Big(\sup_{0< t \leq \delta} \frac{1}{a_N} \left(X ((t+\tau)N^2) - X(\tau N^2)\right) > \varepsilon,\\ 
B_{N,M,\delta}, C_{N,M,\varepsilon,\delta}\Big) = - \infty. 
\end{multline}
Conditioned on the event $B_{N,M,\delta}\bigcap C_{N,M,\varepsilon,\delta}$, we have $X(\tau N^2)=k$ for some $-a_NM\leq k\leq a_NM$. Moreover, if $\frac{1}{a_N} \left(X ((t+\tau)N^2) - X(\tau N^2)\right) > \varepsilon$ for some $0 < t\leq \delta$, then
\begin{align*}
J_{k-1,k}((t+\tau)N^2)-J_{k-1,k}(\tau N^2)&=\sum_{x=k}^{X((t+\tau)N^2)}\eta_{(t+\tau)N^2}(x)\\
&\geq \sum_{x=k}^{k+[a_N\epsilon]}\eta_{(t+\tau)N^2}(x)=\sum_{x=0}^{[a_N\varepsilon]}\eta_{(t+\tau)N^2}(x+k)\geq \frac{\rho a_N\varepsilon}{2}
\end{align*}
for some $-a_NM\leq k\leq a_NM$. 
Thus, to prove \eqref{et11}, we only need to show that
\begin{multline}\label{et12}
\limsup_{\delta \rightarrow 0} \limsup_{N \rightarrow \infty}  \sup_{\tau \in \mathcal{T}_T} \frac{N}{a_N^2} \log \sum_{k:|k|\leq a_NM} \P_{\nu_{\rho}^*} \Big(\sup_{0< t \leq \delta} \frac{1}{a_N}\left(J_{k-1,k}((t+\tau)N^2)-J_{k-1,k}(\tau N^2)\right)\geq \frac{\rho \varepsilon}{2} \Big) \\
= - \infty,
\end{multline}
which follows immediately from \eqref{et2} since $a_N \gg \sqrt{N \log N}$. This concludes the proof.
\end{proof}

\section{Proof of Theorems \ref{theorem MDP of current} and \ref{theorem MDP of tagged particle}}\label{section proofs}

In this section, we prove Theorems \ref{theorem MDP of current} and \ref{theorem MDP of tagged particle}.  We first state two lemmas concerning  finite-dimensional moderate deviation principles for the current and the tagged particle.

\begin{lemma}\label{lemma finite dimensional MDP of current}
For any $n \geq 1$ and for any $0\leq t_1<t_2<\ldots<t_n$, the sequence \[\Big\{\frac{1}{a_N}\left(J_{-1,0}(t_1N^2),\ldots,J_{-1,0}(t_nN^2)\right)\Big\}_{N \geq 1}\] 
satisfies moderate deviation principles with decay rate $a_N^2/N$ and with rate function $\sigma_J^{-2} \mathcal{I}_{_{\{t_j\}_{j=1}^n}}$ under both $\P_{\nu_\rho}$ and $\P_{\nu_\rho^*}$.  Recall the definition of  $\mathcal{I}_{_{\{t_j\}_{j=1}^n}}$ from \eqref{I_tj_n}. More precisely,  for any closed set $\Ccal \subseteq \mathbb{R}^n$,
\begin{equation}\label{equ 5.1 current fdLDP upper bound}
\limsup_{N\rightarrow +\infty}\frac{N}{a_N^2}\log \P\left(\frac{1}{a_N}\left(J_{-1,0}(t_1N^2),\ldots,J_{-1,0}(t_nN^2)\right)\in \Ccal \right)
\leq -\inf_{\alphabf \in \Ccal}\frac{1}{\sigma_J^2}\mathcal{I}_{_{\{t_j\}_{j=1}^n}}(\alphabf),
\end{equation}
and for any open set $\Ocal \subseteq \mathbb{R}^n$, 
\begin{equation}\label{equ 5.2 current fdLDP lower bound}
\liminf_{N\rightarrow +\infty}\frac{N}{a_N^2}\log \P\left(\frac{1}{a_N}\left(J_{-1,0}(t_1N^2),\ldots,J_{-1,0}(t_nN^2)\right)\in \Ocal\right)
\geq -\inf_{\alphabf \in \Ocal}\frac{1}{\sigma_J^2}\mathcal{I}_{_{\{t_j\}_{j=1}^n}}(\alphabf),
\end{equation}
where $\P=\P_{\nu_\rho}$ or $\P=\P_{\nu_\rho^*}$. 
\end{lemma}

\begin{lemma}\label{lemma finite dimensional MDP of tagged particle}
For any $n \geq 1$ and for any $0\leq t_1<t_2<\ldots<t_n$, the sequence \[\Big\{\frac{1}{a_N}\left(X (t_1N^2),\ldots,X (t_nN^2)\right)\Big\}_{N \geq 1}\] 
satisfies moderate deviation principles with decay rate $a_N^2/N$ and with rate function $\sigma_X^{-2} \mathcal{I}_{_{\{t_j\}_{j=1}^n}}$.   More precisely,  for any closed set $\Ccal \subseteq \mathbb{R}^n$,
\begin{equation}\label{equ 5.3 tagged fdLDP upper bound}
\limsup_{N\rightarrow +\infty}\frac{N}{a_N^2}\log \P_{\nu_\rho^*}\left(\frac{1}{a_N}\left(X(t_1N^2),\ldots,X(t_nN^2)\right)\in \Ccal \right)
\leq -\inf_{\alphabf \in \Ccal}\frac{1}{\sigma_X^2}\mathcal{I}_{_{\{t_j\}_{j=1}^n}}(\alphabf),
\end{equation}
and for any open set $\Ocal\subseteq \mathbb{R}^n$, 
\begin{equation}\label{equ 5.4 tagged fdLDP lower bound}
\liminf_{N\rightarrow +\infty}\frac{N}{a_N^2}\log \P_{\nu_\rho^*}\left(\frac{1}{a_N}\left(X(t_1N^2),\ldots,X(t_nN^2)\right)\in \Ocal\right)
\geq -\inf_{\alphabf\in \Ocal}\frac{1}{\sigma_X^2}\mathcal{I}_{_{\{t_j\}_{j=1}^n}}(\alphabf).
\end{equation}
\end{lemma}

We shall prove Lemmas \ref{lemma finite dimensional MDP of current} and \ref{lemma finite dimensional MDP of tagged particle} in later subsections. Now we utilize these two lemmas to prove Theorems \ref{theorem MDP of current} and \ref{theorem MDP of tagged particle}.

\begin{proof}[Proof of Theorems \ref{theorem MDP of current} and \ref{theorem MDP of tagged particle}]

We first prove Theorem \ref{theorem MDP of current}. In Lemma \ref{lem:current exponential tight}, we have shown that the sequence $\{\frac{1}{a_N}J_{-1,0}(tN^2): 0\leq t\leq T\}_{N\geq 1}$ is exponentially tight. Using  \cite[Theorem 4.28]{Feng2006LDP} and Lemma \ref{lemma finite dimensional MDP of current},  the sequence 
$\{\frac{1}{a_N}J_{-1,0}(tN^2): 0\leq t\leq T\}_{N\geq 1}$ satisfies moderate deviation principles with decay rate $a_N^2/N$ and with rate function given by
\[
\sup\Big\{\frac{1}{\sigma_J^2}\mathcal{I}_{_{\{t_j\}_{j=1}^n}}\left(f_{_{\{t_j\}_{j=1}^n}}\right):~n\geq 1, 0\leq t_1<t_2<\ldots<t_n\leq T, t_j\in \Delta_f^c\text{~for all~}1\leq j\leq n\Big\}
\]
for all $f\in \mathcal{D}([0, T],\R)$, where $f_{_{\{t_j\}_{j=1}^n}}$ and $\Delta_f^c$ are defined as in Proposition \ref{prop:rateF_varF_1}. Using Proposition \ref{prop:rateF_varF_1} again, the last supremum equals \[\frac{1}{\sigma_J^2}\mathcal{I}_{path} (f) = \Ical_{cur} (f),\] 
which concludes the proof of Theorem \ref{theorem MDP of current}. For Theorem  \ref{theorem MDP of tagged particle}, we use Lemmas \ref{lem:tagged particle exponential tight} and \ref{lemma finite dimensional MDP of tagged particle} instead. 
\end{proof}

\subsection{Proof of Lemma \ref{lemma finite dimensional MDP of current}}  For $\alphabf \in \R^n$ and $r>0$, let
\[
B(\alphabf,r)=\left\{\betabf \in \mathbb{R}^n:~\|\betabf-\alphabf\|_\infty<r\right\}
\]
be the cube of radius $r$ centered at $\alphabf$. For any $M>0$, shorten
\[\bar{B}_M =\overline{B(\mathbf{0},M)} = \left\{ \alphabf \in \mathbb{R}^n:~\|\alphabf\|_\infty\leq M\right\}.\]
We  first prove the lower bound \eqref{equ 5.2 current fdLDP lower bound}.

\begin{proof}[Proof of \eqref{equ 5.2 current fdLDP lower bound}]
For any open set $\Ocal \subseteq \R^n$ and for any point $\alphabf \in \Ocal$, using \eqref{uniform_see} and \eqref{et2.2},  for sufficiently small $\varepsilon > 0$ such that $B(\alphabf,\varepsilon) \subseteq \Ocal$,
\begin{multline}\label{equ 5.2.1}
\liminf_{N\rightarrow +\infty}\frac{N}{a_N^2}\log \P_{\nu_{\rho}} \left(\frac{1}{a_N}\left(J_{-1,0}(t_1N^2),\ldots,J_{-1,0}(t_nN^2)\right)\in \Ocal\right)\\
\geq \liminf_{m\rightarrow+\infty}\liminf_{N\rightarrow +\infty}\frac{N}{a_N^2}\log \P_{\nu_{\rho}} \Big(\left\{\<\mu_{t_j}^N, G_m\>-\<\mu_{0}^N, G_m\>\right\}_{j=1}^n\in B(\alphabf,\varepsilon/2),\\
\left|\frac{1}{mNa_N} \sum_{x=0}^{mN-1} J_{x,x+1} (t_j N^2)\right|<\varepsilon/2\text{~for all~}1\leq j\leq n\Big)\\
=\liminf_{m\rightarrow+\infty}\liminf_{N\rightarrow +\infty}\frac{N}{a_N^2}\log \P_{\nu_{\rho}} \left(\left\{\<\mu_{t_j}^N, G_m\>-\<\mu_{0}^N, G_m\>\right\}_{j=1}^n\in B(\alphabf,\varepsilon/2)\right).
\end{multline}
 Note that $G_m \notin \Scal (\R)$.  By \eqref{smooth out G_n} and Theorem \ref{thm:mdphl},  the last line is bounded from below by 
\begin{multline*}
\liminf_{m\rightarrow+\infty} \liminf_{N\rightarrow +\infty}\frac{N}{a_N^2}\log \P_{\nu_{\rho}} \left(\left\{\<\mu_{t_j}^N, \tilde{G}_m\>-\<\mu_{0}^N, \tilde{G}_m\>\right\}_{j=1}^n\in B(\alphabf,\varepsilon/4)\right)\\
\geq - \limsup_{m\rightarrow+\infty} \inf\Big\{\Qcal_{t_n}(\mu):~\left\{\<\mu_{t_j}, \tilde{G}_m\>-\<\mu_{0}, \tilde{G}_m\>\right\}_{j=1}^n\in B(\alphabf,\varepsilon/4)\Big\}.
\end{multline*}

Let $\mu^n_{t_n} := \mu_{t_n}^{n,\{t_j\}_{j=1}^n,\{\alpha_j\}_{j=1}^n}$ be the minimizer of the following variational problem
\[		\inf \Big\{\Qcal_{t_n} (\mu): \mu \in \Acal, \;\int_0^\infty [\mu(t_j,u) - \mu(0,u)] du= \alpha_j,\; 1 \leq j \leq n \Big\} \]
which was introduced in  the proof of Proposition \ref{prop:Q_finiteD}. Since $\mu^n_{t_n} \in \Acal$, by the explicit expression of $\mu$ given in \eqref{mu formula},  $\mu_{t_n}^{n} (t, u) \in L^1(\mathbb{R})$ for any $t\in [0, t_n]$. Thus, by dominated convergence theorem, for all $1\leq j\leq n$,
	\begin{equation}\label{pf main 1}
		\lim_{m\rightarrow+\infty} \<\mu^n_{t_n} (t_j,\cdot) - \mu^n_{t_n}(0,\cdot), \tilde{G}_m\>
		= \int_0^\infty \Big[ \mu^n_{t_n} (t_j,u)-\mu^n_{t_n}(0,u) \Big]du =\alpha_j.
	\end{equation}
This implies that, for $m$ large enough,
\[
\left\{\<\mu^n_{t_n} (t_j), \tilde{G}_m\>-\<\mu^n_{t_n} (0), \tilde{G}_m\>\right\}_{j=1}^n\in B(\alphabf,\varepsilon/4).
\]
Consequently, 
\begin{multline*}
\liminf_{N\rightarrow +\infty}\frac{N}{a_N^2}\log \P_{\nu_{\rho}} \left(\frac{1}{a_N}\left(J_{-1,0}(t_1N^2),\ldots,J_{-1,0}(t_nN^2)\right)\in \Ocal \right)\\
\geq -  \Qcal_{t_n}(\mu^n_{t_n})=-\frac{\sqrt{2\pi}}{4\chi(\rho)}\alphabf^TA^{-1}_{\{t_j\}_{j=1}^n}\alphabf
=-\frac{1}{\sigma_J^2}\mathcal{I}_{\{t_j\}_{j=1}^n}(\alphabf).
\end{multline*}
Since $\alphabf \in \Ocal$ is arbitrary,
\begin{align*}
\liminf_{N\rightarrow +\infty}\frac{N}{a_N^2}\log \P_{\nu_{\rho}}\left(\frac{1}{a_N}\left(J_{-1,0}(t_1N^2),\ldots,J_{-1,0}(t_nN^2)\right)\in \Ocal \right)&\geq \sup_{\alphabf \in \Ocal}-\frac{1}{\sigma_J^2}\mathcal{I}_{\{t_j\}_{j=1}^n}(\alphabf)\\
&=-\inf_{\alphabf \in \Ocal}\frac{1}{\sigma_J^2}\mathcal{I}_{\{t_j\}_{j=1}^n}(\alphabf).
\end{align*}
By Remark \ref{rmk:coupling}, the lower bound also holds under $\P_{\nu_{\rho}^*}$, thus concluding the proof.
\end{proof}

Now, we prove the upper bound \eqref{equ 5.1 current fdLDP upper bound}.  By exponential tightness of the current (see Lemma \ref{lem:current exponential tight}), we only need to prove the estimate for any compact subset in $\R^n$. 

\begin{proof}[Proof of \eqref{equ 5.1 current fdLDP upper bound}]
As in the proof of the lower bound, we only need to deal with the case where $\P=\P_{\nu_\rho}$.
Let $\mathcal{K} \subseteq \R^n$ be any compact subset.  For any $r > 0$, we could find a  sequence of points $\alphabf^j \in \mathcal{K}$, $j = 1, 2, \ldots, \ell (r)$, such that
\[\mathcal{K} \subseteq \bigcup_{j=1}^{\ell} B(\alphabf^j,r).\]
By \eqref{et2.2}, for any $m > 0$ and any $r > 0$, 
\begin{multline*}
\limsup_{N\rightarrow +\infty}\frac{N}{a_N^2}\log \P_{\nu_\rho} \left(\frac{1}{a_N}\left(J_{-1,0}(t_1N^2),\ldots,J_{-1,0}(t_nN^2)\right)\in \mathcal{K}\right)\\
\leq \max_{1 \leq j \leq \ell} \limsup_{N\rightarrow +\infty} \frac{N}{a_N^2}\\
\log \P_{\nu_\rho} \left( \big\{\<\mu^N_{t_i},G_m\> -   \<\mu^N_{0},G_m\> + \frac{1}{mNa_N} \sum_{x=0}^{mN-1} J_{x,x+1} (t_i N^2)\big\}_{i=1}^n \in B(\alphabf^j,r)\right),
\end{multline*}
Using Lemma \ref{lem:superexponential_estimate},  the right-hand side of the last inequality is bounded from above by, for any $r > 0$, 
\[\limsup_{m \rightarrow \infty} \max_{1 \leq j \leq \ell} \limsup_{N\rightarrow +\infty} \frac{N}{a_N^2} \log \P_{\nu_\rho} \left( \big\{\<\mu^N_{t_i},G_m\> -   \<\mu^N_{0},G_m\>\big\}_{i=1}^n  \in B(\alphabf^j,2r)\right).\]
By \eqref{smooth out G_n}, we could replace the above $G_m$ with the Schwartz function $\tilde{G}_m$. Then, using moderate deviation principles from hydrodynamic limits (see Theorem \ref{thm:mdphl}), we finally bound the last expression from above by, for any $r > 0$, 
\[\limsup_{m \rightarrow \infty} \max_{1 \leq j \leq \ell}  - \inf \big\{ \Qcal_{t_n} (\mu):  \big \{\<\mu_{t_i},\tilde{G}_m\> -   \<\mu_{0},\tilde{G}_m\>\big\}_{i=1}^n  \in  B(\alphabf^j,3r) \big\}.\]

We first show that there exists some constant $C_0$ independent of $r,j,m$ such that, for $m$ large enough,
\begin{equation}\label{bound infimum}
\inf \big\{ \Qcal_{t_n} (\mu):  \big \{\<\mu_{t_i},\tilde{G}_m\> -   \<\mu_{0},\tilde{G}_m\>\big\}_{i=1}^n  \in  B(\alphabf^j,3r) \big\} \leq C_0.
\end{equation}
Indeed, let  $\mu^{n,j}_{t_n} := \mu_{t_n}^{n,\{t_i\}_{i=1}^n,\{\alpha^j_i\}_{i=1}^n}$ be as in the proof of \eqref{equ 5.2 current fdLDP lower bound},   where $\alpha^j_i$ is the $i$-th component of $\alphabf^j$ for $1 \leq i \leq n$.  By \eqref{pf main 1}, for $m$ large enough,
\[\big\{\<\mu^{n,j}_{t_n} (t_i)- \mu^{n,j}_{t_n} (0),\tilde{G}_m\> \big\}_{i=1}^n  \in B(\alphabf^j,3r).\]
Therefore, the infimum in \eqref{bound infimum} is bounded from above by 
\[\Qcal_{t_n} (\mu^{n,j}_{t_n}) = \frac{\sqrt{2\pi}}{4 \chi (\rho)} (\alphabf^j )^T A_{\{t_i\}_{i=1}^n}^{-1} \alphabf^j \leq \sup_{\alphabf \in \mathcal{K}} \frac{\sqrt{2\pi}}{4 \chi (\rho)} \alphabf^T A_{\{t_i\}_{i=1}^n}^{-1} \alphabf := C_0,\]
which proves \eqref{bound infimum}. 

For any $r,j,m$ and any $\varepsilon > 0$, there exists $\mu^{r,j,m}_\varepsilon$ such that
\begin{equation}\label{current_condition_1}
\big\{\<\mu^{r,j,m}_\varepsilon (t_i,\cdot) - \mu^{r,j,m}_\varepsilon (0,\cdot),\tilde{G}_m\> \big\}_{i=1}^n   \in B(\alphabf^j,3r),
\end{equation}
and
\begin{equation*}
\Qcal_{t_n} (\mu^{r,j,m}_\varepsilon) - \varepsilon \leq \inf \big\{ \Qcal_{t_n} (\mu):  \big \{\<\mu_{t_i},\tilde{G}_m\> -   \<\mu_{0},\tilde{G}_m\>\big\}_{i=1}^n  \in  B(\alphabf^j,3r) \big\}.
\end{equation*}
Moreover, by \eqref{bound infimum},
\begin{equation}\label{bound_Q_rjm}
	\sup_{r,j,m} \Qcal_{t_n} (\mu^{r,j,m}_\varepsilon)  \leq C_0 + \varepsilon.
\end{equation}
So far, we have shown that, for any $\varepsilon > 0$,
\begin{multline}\label{Q_limit}
\limsup_{N\rightarrow +\infty}\frac{N}{a_N^2}\log \P_{\nu_\rho} \left(\frac{1}{a_N}\left(J_{-1,0}(t_1N^2),\ldots,J_{-1,0}(t_nN^2)\right)\in \mathcal{K}\right) \\
\leq - \liminf_{ r \rightarrow 0} \liminf_{m\rightarrow+\infty} \min_{1 \leq j \leq \ell} \Qcal_{t_n} (\mu^{r,j,m}_\varepsilon) + \varepsilon = - \liminf_{ r \rightarrow 0} \liminf_{m\rightarrow+\infty} \Qcal_{t_n} (\mu^{r,\tilde{j},m}_\varepsilon) + \varepsilon,
\end{multline} 
where $\tilde{j}:=j (r,m,\varepsilon) := \arg  \min_{1 \leq j \leq \ell} \Qcal_{t_n} (\mu^{r,j,m}_\varepsilon)$. 

By \eqref{bound_Q_rjm}, we could extract a subsequence along which the limit on the right-hand side of  \eqref{Q_limit} is attained as $r \rightarrow 0, m \rightarrow +\infty$. According to Remark \ref{rmk:goodness of Qcal}, $\{\mu^{r,j,m}_\varepsilon:~r,j,m\}$ is a subset of the compact set $\mathbb{L}_{C_0+\varepsilon}$. Hence, by further extracting a subsequence, there exists $\mu_\varepsilon^*$ such that $\mu^{r,\tilde{j},m}_\varepsilon$ converges to $\mu_\varepsilon^*$ in $\mathcal{D} ([0,T],\Scal^\prime (\R))$, and  by the compactness of $\mathcal{K}$, there also exists some $\alphabf = (\alpha_1, \alpha_2, \ldots, \alpha_n) \in \mathcal{K}$ such that $\alphabf^{\tilde{j}}$ converges to $\alphabf$  along this subsequence. Then, by the lower-semicontinuity of $\Qcal_{t_n}$,  
\[\Qcal_{t_n} (\mu^*_\varepsilon) \leq \liminf_{ r \rightarrow 0} \liminf_{m\rightarrow+\infty} \Qcal_{t_n} (\mu^{r,\tilde{j},m}_\varepsilon),\]
and thus,
\begin{equation}\label{pf main_2}
\limsup_{N\rightarrow +\infty}\frac{N}{a_N^2}\log \P_{\nu_\rho} \left(\frac{1}{a_N}\left(J_{-1,0}(t_1N^2),\ldots,J_{-1,0}(t_nN^2)\right)\in \mathcal{K}\right) \leq - \Qcal_{t_n} (\mu^*_\varepsilon) + \varepsilon.
\end{equation}
Moreover, by Lemma \ref{lem:current condition approxiamation} below,
\begin{equation}\label{current_condition}
\int_{0}^\infty \big[\mu^*_\varepsilon (t_i,u) - \mu^*_\varepsilon (0,u) \big]du  = \alpha_i, \quad \forall 1 \leq i \leq n.
\end{equation}
Therefore, by Proposition \ref{prop:rateF_varF_2},
\begin{multline*}
	\limsup_{N\rightarrow +\infty}\frac{N}{a_N^2}\log \P_{\nu_\rho} \left(\frac{1}{a_N}\left(J_{-1,0}(t_1N^2),\ldots,J_{-1,0}(t_nN^2)\right)\in \mathcal{K}\right) \\
	\leq - \inf_{\alphabf \in \mathcal{K}}  \inf_{\mu} \Big\{\Qcal_{t_n} (\mu): \int_{0}^\infty \big[\mu (t_i,u) - \mu  (0,u) \big]du  = \alpha_i, \; \forall 1 \leq i \leq n\Big\} + \varepsilon \\
	= - \inf_{\alphabf \in \mathcal{K}}  \frac{1}{\sigma_J^2} \Ical_{\{t_j\}_{j=1}^n} (\alphabf) + \varepsilon.
\end{multline*}
We conclude the proof of  $\eqref{equ 5.1 current fdLDP upper bound}$ by letting $\varepsilon \rightarrow 0$. 
\end{proof}

It remains to prove the following lemma. 

\begin{lemma}\label{lem:current condition approxiamation}
	If the sequence $\mu^m \in \mathcal{D} ([0,T], \Scal' (\R))$, $m \geq 1$, satisfies that
	\[\lim_{m\rightarrow\infty} \mu^m = \mu\; \text{in $\mathcal{D} ([0,T],\mathcal{S}^\prime (\R))$}, \quad \sup_{m \geq 1} \Qcal_T (\mu_m) \leq C_0\]
	for some $\mu \in \mathcal{D} ([0,T], \Scal' (\R))$ and for some finite constant $C_0$, then for any $0 < t \leq T$,
	\[\lim_{m\rightarrow\infty} \int_\R [\mu^m(t,u) - \mu^m(0,u)] \tilde{G}_m (u)\,du = \int_0^\infty [\mu (t,u) - \mu (0,u)] \,du.\]
\end{lemma}

\begin{proof}
For any $m \geq n$, we write
	\begin{multline}\label{current approx 1}
		\int_\R [\mu^m(t,u) - \mu^m(0,u)] \tilde{G}_m (u)\,du \\
		=  	\int_\R [\mu^m(t,u) - \mu^m(0,u)] [\tilde{G}_m (u) - \tilde{G}_n (u)] \,du  + 	\int_\R [\mu^m(t,u) - \mu^m(0,u)] \tilde{G}_n (u) \,du.
	\end{multline}

We first prove that the first term on the right-hand side in the last identity converges to zero as $m \rightarrow \infty, n \rightarrow \infty$. By Lemma \ref{lem:mu property-1}, Eqn.\,\eqref{mu formula} and the uniform boundedness of $\Qcal_T (\mu_m)$, there exists some constant $C_1$ such that, for any $t > 0$, 
\[\sup_{m \geq 1} \|\mu^m(t,\cdot)\|_{L^2 (\R)} \leq C_1.\]
Since $\|G_m - \tilde{G}_m\|^2_{L^2 (\R)} \leq C m^{-1}$, by Cauchy-Schwarz inequality, we only need to show that
\[\lim_{n \rightarrow \infty}\,\lim_{m\rightarrow\infty}\, \int_\R [\mu^m(t,u) - \mu^m(0,u)] [G_m (u) - G_n (u)] \,du = 0.\]
We could also find $\tilde{\mu}^m \in \Acal$ such that $\|\tilde{\mu}^m (t,\cdot) - \mu^m (t, \cdot)\|_{L^2 (\R)}^2 \leq C m^{-2}$. Since $\|G_m \|^2_{L^2 (\R)} \leq C m$, using Cauchy-Schwarz inequality again, it is sufficient to show that
\begin{equation}\label{current approx 2}
\lim_{n \rightarrow \infty}\,\lim_{m\rightarrow\infty}\, \int_\R [\tilde{\mu}^m(t,u) - \tilde{\mu}^m(0,u)] [G_m (u) - G_n (u)] \,du = 0.
\end{equation}
Let $\tilde{J}^m$ be the current associated with $\tilde{\mu}^m$, \emph{i.e.}, $\partial_{t} \tilde{\mu}^m + \partial_{u} \tilde{J}^m = 0$. Since 
\[G_m (u) - G_n (u) = \begin{cases}
	0 &\quad \text{if} \quad u \leq 0\;\text{or}\; u \geq m;\\ 
	\frac{u}{n} - \frac{u}{m} &\quad \text{if} \quad 0 \leq u \leq n;\\
	1 - \frac{u}{m} &\quad \text{if} \quad n \leq u \leq m,
\end{cases}\]
using integration by parts formula, we rewrite the left-hand side in \eqref{current approx 2} as
\[\Big(\frac{1}{n} - \frac{1}{m}\Big) \int_{0}^t \int_0^n \tilde{J}^m (s,u) du ds - \frac{1}{m}\int_{0}^t \int_n^m \tilde{J}^m (s,u) du ds.\]
Note that by \eqref{J} and the uniform boundedness of $\Qcal_T (\mu_m)$,  there exists some constant $C_2$ such that $\sup_{0 \leq t \leq T} \sup_{m \geq 1} \|\tilde{J}^m (t,\cdot)\|_{L^2 (\R)} \leq C_2$. Then,  by Cauchy-Schwarz inequality, the last expression is bounded by $C (n^{-1/2} + m^{-1/2})$, concluding the proof of  \eqref{current approx 2}. 

For the second term in \eqref{current approx 1}, letting $m \rightarrow \infty$, it converges to
\[\int_\R [\mu(t,u) - \mu (0,u)] \tilde{G}_n (u) du.\]
Note that $\mu(t,\cdot)$ may not be in $L^1 (\R)$, and thus we cannot use dominated convergence theorem directly.  Instead, we first replace $\tilde{G}_n (u)$ with $G_n (u)$ as in the above argument, and then only need to deal with 
\[\int_\R [\mu(t,u) - \mu (0,u)] G_n (u) du = \int_0^n [\mu(t,u) - \mu (0,u)] du - \frac{1}{n}\int_0^n u [\mu(t,u) - \mu (0,u)] \,du\]
by the definition of $G_n$.  At last, letting $n \rightarrow \infty$ and using Lemma \ref{lem:mu property-2} and Remark \ref{rmk:mu property}, we conclude  the proof.
\end{proof}

\subsection{Proof of Lemma \ref{lemma finite dimensional MDP of tagged particle}} We start with the lower bound \eqref{equ 5.4 tagged fdLDP lower bound}. 

\begin{proof}[Proof of \eqref{equ 5.4 tagged fdLDP lower bound}]
Observe that if $J_{-1,0}(t)\geq 0$, then
\begin{equation}\label{equ 5.3.1}
J_{-1,0}(t)=\sum_{x=0}^{X(t)-1}\eta_t(x)=\sum_{x=0}^{X(t)-1}\left(\eta_t(x)-\rho\right)+\rho X(t),
\end{equation}
and if $J_{-1,0}(t)<0$, then 
\begin{equation}\label{equ 5.3.2}
J_{-1,0}(t)=-\sum_{x=X(t)}^{-1}\eta_t(x)=-\sum_{x=X(t)}^{-1}\left(\eta_t(x)-\rho\right)+\rho X(t).
\end{equation}
For any given $\alphabf=(\alpha_1,\ldots,\alpha_n)^T\in \Ocal$ and any $\varepsilon>0$, by \eqref{equ 5.2 current fdLDP lower bound},
\begin{align}\label{equ 5.3.3}
&\liminf_{N\rightarrow+\infty}\frac{N}{a_N^2}\log\P_{\nu_\rho^*} \left(\frac{1}{a_N}\left(J_{-1,0}(t_1N^2),\ldots,J_{-1,0}(t_nN^2)\right)\in B(\rho\alphabf, \varepsilon\rho/10)\right)\\
&\geq -\inf_{\betabf \in B(\rho\alphabf , \varepsilon\rho/10)}\frac{1}{\sigma_J^2}\mathcal{I}_{_{\{t_j\}_{j=1}^n}}(\betabf)>-\infty. \notag
\end{align}
For $M>0$, let $D_{M,N}$ be the event that $|X(t_jN^2)|\leq a_NM$ for all $1\leq j\leq n$. Using \eqref{et7}, we have
\[
\limsup_{M\rightarrow+\infty}\limsup_{N\rightarrow+\infty}\frac{N}{a_N^2}\log\P_{\nu_{\rho}^*} \left(D_{M,N}^c\right)=-\infty.
\]
Hence, there exists $M_1>0$ such that
\begin{multline}\label{equ 5.3.4}
\liminf_{N\rightarrow+\infty}\frac{N}{a_N^2}\log\P_{\nu_{\rho}^*} \left(\frac{1}{a_N}\left(J_{-1,0}(t_1N^2),\ldots,J_{-1,0}(t_nN^2)\right)\in B(\rho\alphabf, \varepsilon\rho/10), D_{M_1,N}\right) \\
=\liminf_{N\rightarrow+\infty}\frac{N}{a_N^2}\log\P_{\nu_{\rho}^*} \left(\frac{1}{a_N}\left(J_{-1,0}(t_1N^2),\ldots,J_{-1,0}(t_nN^2)\right)\in B(\rho\alphabf, \varepsilon\rho/10)\right) \\
\geq -\inf_{\betabf\in B(\rho\alphabf, \varepsilon\rho/10)}\frac{1}{\sigma_J^2}\mathcal{I}_{_{\{t_j\}_{j=1}^n}}(\betabf).
\end{multline}
Let $F_{N}$ be the event that $|\frac{1}{a_N}\sum_{x=0}^{k}(\eta_{t_jN^2} (x)-\rho)|<\frac{\rho\varepsilon}{10}$
and $|\frac{1}{a_N}\sum_{x=-k}^{-1}(\eta_{t_jN^2} (x)-\rho)|<\frac{\rho\varepsilon}{10}$ for all $1\leq k\leq M_1a_N$ and all $1\leq j\leq n$. By Remark \ref{rmk:coupling} and standard large deviation results, we have
\[
\limsup_{N\rightarrow+\infty}\frac{1}{a_N}\log \P_{\nu_\rho^*}\left(F_{N}^c\right)<0
\]
and thus
\begin{equation}\label{equ 5.3.5}
 \limsup_{N\rightarrow+\infty}\frac{N}{a_N^2}\log  \P_{\nu_\rho^*} \left(F_{N}^c\right)=-\infty.   
\end{equation}
Using \eqref{equ 5.3.4} and \eqref{equ 5.3.5}, we have
\begin{align}\label{equ 5.3.6}
  &\liminf_{N\rightarrow+\infty}\frac{N}{a_N^2}\log\P_{\nu_\rho^*}  \left(\frac{1}{a_N}\left(J_{-1,0}(t_1N^2),\ldots,J_{-1,0}(t_nN^2)\right)\in B(\rho\alphabf, \varepsilon\rho/10), D_{M_1,N}, F_{N}\right) \notag\\  
  &\geq -\inf_{\betabf\in B(\rho\alphabf, \varepsilon\rho/10)}\frac{1}{\sigma_J^2}\mathcal{I}_{_{\{t_j\}_{j=1}^n}}(\betabf).
\end{align}
Conditioned on the event $D_{M_1,N}\bigcap F_{N}$, according to \eqref{equ 5.3.1} and \eqref{equ 5.3.2}, if \[\frac{1}{a_N}\left(J_{-1,0}(t_1N^2),\ldots,J_{-1,0}(t_nN^2)\right)\in B(\rho\alphabf, \varepsilon\rho/10),\] 
then
\[
\left|\frac{1}{a_N}X(t_jN^2)-\alpha_j\right|<\varepsilon/5<\varepsilon, \quad \forall 1\leq j\leq n.
\]
Consequently, for sufficiently small $\varepsilon > 0$ such that $B(\alphabf,\varepsilon)\subseteq \Ocal$, we have
\begin{align*}
&\liminf_{N\rightarrow +\infty}\frac{N}{a_N^2}\log \P_{\nu_\rho^*}  \left(\frac{1}{a_N}\left(X(t_1N^2),\ldots,X(t_nN^2)\right)\in \Ocal\right)\\
&\geq \liminf_{N\rightarrow +\infty}\frac{N}{a_N^2}\log \P_{\nu_\rho^*}  \left(\frac{1}{a_N}\left(X(t_1N^2),\ldots,X(t_nN^2)\right)\in B(\alphabf, \varepsilon)\right)\\
&\geq \liminf_{N\rightarrow+\infty}\frac{N}{a_N^2}\log\P_{\nu_\rho^*} \left(\frac{1}{a_N}\left(J_{-1,0}(t_1N^2),\ldots,J_{-1,0}(t_nN^2)\right)\in B(\rho\alphabf, \varepsilon\rho/10), D_{M_1,N}, F_{N}\right)\\
&\geq -\inf_{\betabf\in B(\rho\alphabf, \varepsilon\rho/10)}\frac{1}{\sigma_J^2}\mathcal{I}_{_{\{t_j\}_{j=1}^n}}(\betabf)
\end{align*}
Since $\varepsilon$ is arbitrary and $\mathcal{I}_{_{\{t_j\}_{j=1}^n}}$ is continuous, let $\varepsilon\rightarrow 0$, we have
\[
\liminf_{N\rightarrow +\infty}\frac{N}{a_N^2}\log \P_{\nu_\rho^*}  \left(\frac{1}{a_N}\left(X(t_1N^2),\ldots,X(t_nN^2)\right)\in \Ocal \right)\geq -\frac{1}{\sigma_J^2}\mathcal{I}_{_{\{t_j\}_{j=1}^n}}(\rho\alphabf)=-\frac{1}{\sigma_X^2}\mathcal{I}_{_{\{t_j\}_{j=1}^n}}(\alphabf).
\]
Since $\alphabf \in \Ocal$ is arbitrary, 
\begin{align*}
\liminf_{N\rightarrow +\infty}\frac{N}{a_N^2}\log \P_{\nu_\rho^*}  \left(\frac{1}{a_N}\left(X(t_1N^2),\ldots,X(t_nN^2)\right)\in \Ocal\right)&\geq \sup_{\alphabf\in \Ocal}\left(-\frac{1}{\sigma_X^2}\mathcal{I}_{_{\{t_j\}_{j=1}^n}}(\alphabf)\right)\\
&=-\inf_{\alphabf \in \Ocal}\frac{1}{\sigma_X^2}\mathcal{I}_{_{\{t_j\}_{j=1}^n}}(\alphabf)
\end{align*}
and the proof is completed.
\end{proof}

At last, we prove the upper bound \eqref{equ 5.3 tagged fdLDP upper bound}.

\begin{proof}[Proof of \eqref{equ 5.3 tagged fdLDP upper bound}]

According to Lemma \ref{lem:tagged particle exponential tight}, we only need to prove \eqref{equ 5.3 tagged fdLDP upper bound} for all compact set
$\mathcal{K} \subseteq \mathbb{R}^n$. For any compact set $\mathcal{K} \subseteq \mathbb{R}^n$, there exists $M_2>0$ such that $\mathcal{K} \subseteq \bar{B}_{M_2}$. For given $0<\varepsilon<1$ and each $N\geq 1$, we denote by $V_N$ the event that $|\frac{1}{a_N}\sum_{x=0}^{k}\left(\eta_{t_jN^2} (x)-\rho\right)|\leq \varepsilon$ and $|\frac{1}{a_N}\sum_{x=-k}^{-1}\left(\eta_{t_jN^2} (x)-\rho\right)|\leq \varepsilon$ for all $1\leq k\leq (M_2+1)a_N$ and $1\leq j\leq n$. Then, similar with \eqref{equ 5.3.5}, we have
\begin{equation*}
\limsup_{N\rightarrow+\infty}\frac{N}{a_N^2}\log\P_{\nu_\rho^*}(V_N^c)=-\infty
\end{equation*}
and thus
\begin{multline}\label{equ 5.3.7}
\limsup_{N\rightarrow +\infty}\frac{N}{a_N^2}\log \P_{\nu_\rho^*}\left(\frac{1}{a_N}\left(X(t_1N^2),\ldots,X(t_nN^2)\right)\in \mathcal{K}\right)\\
=\limsup_{N\rightarrow +\infty}\frac{N}{a_N^2}\log \P_{\nu_\rho^*}\left(\frac{1}{a_N}\left(X(t_1N^2),\ldots,X(t_nN^2)\right)\in \mathcal{K}, V_N\right).
\end{multline}
Conditioned on $V_N$, if $\frac{1}{a_N}\left(X(t_1N^2),\ldots,X(t_nN^2)\right)\in \mathcal{K}$, then, by \eqref{equ 5.3.1} and \eqref{equ 5.3.2}, 
\[
\frac{1}{a_N}\left(J_{-1,0}(t_1N^2),\ldots,J_{-1,0}(t_nN^2)\right)\in \left(\rho \mathcal{K}\right)_\varepsilon,
\]
where $\rho \mathcal{K}=\{\rho\betabf:~\betabf\in \mathcal{K}\}$ and for some subset $\Acal \subset \R^n$ and some $a > 0$,
\[
\Acal_a:=\{\betabf \in \mathbb{R}^n:~\|\betabf-\gammabf\|_\infty\leq a \text{~for some~}\gammabf \in \Acal\}.
\]
 Since $\mathcal{K}$ is compact, $(\rho \mathcal{K})_\varepsilon$ is also compact for all $\varepsilon>0$. As a result, by  \eqref{equ 5.1 current fdLDP upper bound}, for  any $0<\varepsilon<1$,
\begin{multline*}
\limsup_{N\rightarrow +\infty}\frac{N}{a_N^2}\log \P_{\nu_\rho^*}\left(\frac{1}{a_N}\left(X(t_1N^2),\ldots,X(t_nN^2)\right)\in \mathcal{K}, V_N\right)\\
\leq \limsup_{N\rightarrow +\infty}\frac{N}{a_N^2}\log \P_{\nu_\rho^*}\left(\frac{1}{a_N}\left(J_{-1,0}(t_1N^2),\ldots,J_{-1,0}(t_nN^2)\right)\in (\rho \mathcal{K})_\varepsilon \right)\\
\leq -\inf_{\betabf\in (\rho \mathcal{K})_\epsilon}\frac{1}{\sigma_J^2}\mathcal{I}_{_{\{t_j\}_{j=1}^n}}(\betabf).
\end{multline*}
Since $\mathcal{I}_{_{\{t_j\}_{j=1}^n}}$ is uniformly continuous on $\bar{B}_{M_2+1}$ and $(\rho K)_\epsilon$ is compact,
\[
\lim_{\epsilon\rightarrow 0}\inf_{\betabf\in (\rho \mathcal{K})_\epsilon}\frac{1}{\sigma_J^2}\mathcal{I}_{_{\{t_j\}_{j=1}^n}}(\betabf)
=\inf_{\betabf\in \rho \mathcal{K}}\frac{1}{\sigma_J^2}\mathcal{I}_{_{\{t_j\}_{j=1}^n}}(\betabf) =\inf_{\alphabf \in \mathcal{K}}\frac{1}{\sigma_X^2}\mathcal{I}_{_{\{t_j\}_{j=1}^n}}(\alphabf).
\]
Therefore, by \eqref{equ 5.3.7},
\[
\limsup_{N\rightarrow +\infty}\frac{N}{a_N^2}\log \P_{\nu_\rho^*} \left(\frac{1}{a_N}\left(X(t_1N^2),\ldots,X(t_nN^2)\right)\in \mathcal{K}\right)
\leq -\inf_{\alphabf \in \mathcal{K}}\frac{1}{\sigma_X^2}\mathcal{I}_{_{\{t_j\}_{j=1}^n}}(\alphabf),
\]
and the proof is completed.

\end{proof}

\appendix

\section{Calculations}\label{sec:calculations}

\subsection{Proof of Lemma \ref{lem:mu property-1}}\label{app:proof lem mu p-1} We only prove the statement for $\Qcal_{T,dyn}$, and the remaining is much simpler. If $\Qcal_{T,dyn} (\mu) < \infty$, then there exists some constant $A > 0$ such that, for any $G \in \Ccal_c^{1,\infty} ([0,T] \times \R)$,
\[l (\mu,G) - \frac{\chi(\rho)}{2} [G,G] \leq A.\]
Replacing $G$ with $aG$ for any $a \in \R$, we have that
\[a l (\mu,G) - a^2 \frac{\chi(\rho)}{2} [G,G] \leq A, \quad \forall a \in \R.\]
Optimizing over $a$,
\[l (\mu,G)^2 \leq 2 \chi(\rho) A [G,G], \quad \forall G \in \Ccal_c^{1,\infty} ([0,T] \times \R).\]
Since $\Ccal_c^{1,\infty} ([0,T] \times \R)$ is dense in $\mathcal{H}^1$, $l (\mu,\cdot)$ is a linear bounded functional on $\mathcal{H}^1$. Thus, there exists some function $H \in \mathcal{H}^1$ such that
\[l (\mu,G) = \chi (\rho) [H,G], \quad \forall G \in \mathcal{H}^1.\]
In particular,
\[\Qcal_{T,dyn} (\mu) = \chi (\rho) \sup_{G \in \Ccal^{1,\infty}_c ( [0,T]  \times \R)} \Big\{ [H,G] - \frac{1}{2} [G,G] \Big\} = \frac{\chi(\rho)}{2} [H,H],\]
which concludes the proof.

\subsection{Proof of \eqref{pre3}}\label{app pf pre3} In this subsection, we prove that, for any $t > 0$,
\[\int_0^t {p}_s^{\prime} (\cdot)\,ds  \in L^2 (\R),\]
where $p_s (\cdot)$ is the heat kernel. By direct calculations, 
\[p_s^\prime (u) = - \frac{u}{s} p_s (u).\]
Making the change of variables $u^2/2s \mapsto \tau$, 
\[\int_0^t {p}_s^{\prime} (\cdot)\,ds = \int_0^t - \frac{u}{\sqrt{2\pi}s^{3/2}} e^{-u^2/(2s)} ds = -\int_{u^2/2t}^{+\infty}  \frac{1}{\sqrt{\pi \tau}} e^{-\tau}\,d\tau. \]
Thus,
\[\Big| \int_0^t {p}_s^{\prime} (u)\,ds \Big| \leq C \min \Big\{ \frac{\sqrt{t}}{|u|} e^{-u^2/(2t)}, 1\Big\}\]
for some finite constant $C$, which concludes the proof.

\subsection{Proof of \eqref{q_mu}}\label{appendix:pf_q_mu}  Recall in Lemma \ref{lem:mu property-1} we have shown that
\[\mathcal{Q}_{T,d y n}(\mu)=\frac{\chi(\rho)}{2} \int_{0}^{T} \int_{\mathbb{R}}\left(\partial_{u} H\right)^{2}(t, u) d u d t.\]
Replacing $\partial_{u} H$ with $\chi(\rho)^{-1}  [J+\tfrac{1}{2} \partial_{u} \mu]$, we have
\[\chi (\rho)\mathcal{Q}_{T,d y n}(\mu)= \int_{0}^{T} \int_{\mathbb{R}}\left\{\frac{1}{2} J^{2}(t, u)+\frac{1}{8}\left(\partial_{u} \mu\right)^{2}(t, u)+\frac{1}{2} J(t, u) \partial_{u} \mu(t, u)\right\} d u d t.\]
Integrating by parts and using $\partial_t \mu + \partial_{u} J = 0$, 
\[ \int_{0}^{T} \int_{\mathbb{R}} \frac{1}{2} J(t, u) \partial_{u} \mu(t, u) d u d t= \frac{1}{4} \int_{\mathbb{R}}\left\{\mu^{2}(T, u)-\mu_{0}^{2}(u)\right\} d u.\]
Thus,
\[\chi (\rho) \mathcal{Q}_{T,d y n}(\mu)=\int_{0}^{T} \int_{\mathbb{R}}\left\{\frac{1}{2} J^{2}(t, u)+\frac{1}{8}\left(\partial_{u} \mu\right)^{2}(t, u)\right\} d u d t+\frac{1}{4} \int_{\mathbb{R}}\left\{\mu^{2}(T, u)-\mu_{0}^{2}(u)\right\} d u.\]
Finally, we obtain \eqref{q_mu} by using \eqref{mu_K} and the formula for $\Qcal_0$ in Lemma \ref{lem:mu property-1}.

\subsection{Proof of \eqref{current_fourier}}\label{appendix:pf_current_fourier} Recall that 	for $0 \leq s \leq t$,
\[\Fcal  K^{t,\alpha}_T (s,\xi) = \frac{\alpha}{2 \sqrt{t} \xi^2} \big[e^{-(t-s) \xi^2/2} - e^{-s\xi^2/2} + 1 - e^{-t\xi^2/2}\big],\]
and for $t \leq s \leq T$, 
\[\Fcal  K^{t,\alpha}_T (s,\xi) =\frac{\alpha}{2 \sqrt{t} \xi^2} \big[e^{-(s-t) \xi^2/2} - e^{-s\xi^2/2} + 1 - e^{-t\xi^2/2}\big].\]
Using the integeral formula
\begin{equation}\label{integral_formula}
\int_{\R} \frac{1-e^{-ax^2}}{x^2} dx = 2 \sqrt{\pi a}, \quad a > 0,
\end{equation}
one easily proves \eqref{current_fourier}.  To prove \eqref{integral_formula}, the left-hand side of the identity equals
\[\int_{\R} \int_0^a e^{-y x^2} dy dx = \int_0^a \frac{\sqrt{\pi}}{\sqrt{y}} dy = 2 \sqrt{\pi a},\]
as claimed.

\subsection{Proof of \eqref{k_n_mu_n}}\label{appendix:k_nmu_n}
We prove \eqref{k_n_mu_n} by calculus of variations.   
For any $\mu_0,\hat{\mu}_0\in \Ccal_c^2(\R)$ and $K, \hat{K}\in \Ccal^{1,2}_c\left([0, T]\times \mathbb{R}\right)$, we denote by $\Lambda\left((\mu_0, K),(\hat{\mu}_0, \hat{K})\right)$ the positive-definite quadratic form
\begin{align*}
	&\int_0^T\int_{\mathbb{R}}\frac{1}{2}\partial_tK(t,u)\partial_t\hat{K}(t,u)dudt
	+\int_0^T\int_{\mathbb{R}}\frac{1}{8}\left(\partial_u\mu_0(u)-\partial_u^2K(t,u)\right)\left(\partial_u\hat{\mu}_0(u)-\partial_u^2\hat{K}(t,u)\right)dudt\\
	&+\int_\mathbb{R}\frac{1}{2}\left(\mu_0(u)-\frac{1}{2}\partial_uK(T,u)\right)\left(\hat{\mu}_0(u)-\frac{1}{2}\partial_u\hat{K}(T,u)\right)du
	+\int_\mathbb{R}\frac{1}{8}\partial_uK(T,u)\partial_u\hat{K}(T,u)du.
\end{align*}
Then, using \eqref{q_mu}, we have
\begin{equation}\label{equ:self inner product}
	\chi(\rho)\Qcal_T(\mu_0,K)=\Lambda\left((\mu_0, K), (\mu_0, K)\right).  
\end{equation}

For any given $\hat{\mu}_0\in \Ccal_c^2(\mathbb{R})$ and $\hat{K}\in \Ccal_c^{1,2}\left([0, T]\times\mathbb{R}\right)$ such that $\hat{K}(0,\cdot)=0$ and $\hat{K}(t_j, 0)=0$, we define
\[
G_j(\epsilon)=\Lambda\left((\mu_{T,0}^{t_j,1}, K_T^{t_j,1})+\epsilon(\hat{\mu}_0,\hat{K}),(\mu_{T,0}^{t_j,1}, K_T^{t_j,1})+\epsilon(\hat{\mu}_0,\hat{K})\right)
\]
for all $\epsilon\in \mathbb{R}$. By Lemma \ref{lem:Q_2}, 
\[
\frac{d}{d\epsilon}G_j(\epsilon)\Big|_{\epsilon=0}=0
\]
and hence
\begin{equation}\label{equ:inner product is zero}
	\Lambda((\mu_{T,0}^{t_j,1}, K_T^{t_j,1}),(\hat{\mu}_0, \hat{K}))=0.
\end{equation}
For any given $\mu_0\in \Ccal_c^2(\mathbb{R})$ and $K\in \Ccal_c^{1,2}\left([0, T]\times\mathbb{R}\right)$ such that $K(0,\cdot)=0$ and $K(t_j,0)=\alpha_j$ for all $1\leq j\leq n$, let
\[
\tilde{K}=K-\sum_{j=1}^{n} \beta_j K_T^{t_j,1}\text{~and~}\tilde{\mu}_0=\mu_0-\sum_{j=1}^{n} \beta_j \mu_{T,0}^{t_j,1},
\]
where $\{\beta_j\}_{j\geq 1}$ satisfy \eqref{equ:constraints}, then
\[
\tilde{K}(0,\cdot)=0 \text{~and~}\tilde{K}(t_j,0)=0
\]
for all $1\leq j\leq n$. Consequently, by \eqref{equ:inner product is zero},
\begin{equation}\label{equ: inner product is zero 2}
	\Lambda\left(\left(\sum_{j=1}^{n} \beta_j \mu_{T,0}^{t_j,1}, \sum_{j=1}^{n} \beta_j K_T^{t_j,1}\right),(\tilde{\mu}_0, \tilde{K})\right)=0.
\end{equation}
For any $\epsilon\in \mathbb{R}$, let
\[
\tilde{G}(\epsilon)=\Lambda\left(\left(\sum_{j=1}^{n} \beta_j \mu_{T,0}^{t_j,1}, \sum_{j=1}^{n} \beta_j K_T^{t_j,1}\right)+\epsilon(\tilde{\mu}_0, \tilde{K}),\left(\sum_{j=1}^{n} \beta_j \mu_{T,0}^{t_j,1}, \sum_{j=1}^{n} \beta_j K_T^{t_j,1}\right)+\epsilon(\tilde{\mu}_0, \tilde{K}) \right).
\]
By \eqref{equ: inner product is zero 2}, for all $\epsilon\in\mathbb{R}$,
\begin{align*}
	\tilde{G}(\epsilon)
	&=\Lambda\left(\left(\sum_{j=1}^{n} \beta_j \mu_{T,0}^{t_j,1}, \sum_{j=1}^{n} \beta_j K_T^{t_j,1}\right), \left(\sum_{j=1}^{n} \beta_j \mu_{T,0}^{t_j,1}, \sum_{j=1}^{n} \beta_j K_T^{t_j,1}\right)\right)+\epsilon^2\Lambda(\tilde{\mu}_0, \tilde{K})\\
	&\geq\Lambda\left(\left(\sum_{j=1}^{n} \beta_j \mu_{T,0}^{t_j,1}, \sum_{j=1}^{n} \beta_j K_T^{t_j,1}\right), \left(\sum_{j=1}^{n} \beta_j \mu_{T,0}^{t_j,1}, \sum_{j=1}^{n} \beta_j K_T^{t_j,1}\right)\right).
\end{align*}
Therefore,
\[
\Lambda\left((\mu_0, K), (\mu_0, K)\right)=\tilde{G}(1)\geq \Lambda\left(\left(\sum_{j=1}^{n} \beta_j \mu_{T,0}^{t_j,1}, \sum_{j=1}^{n} \beta_j K_T^{t_j,1}\right), \left(\sum_{j=1}^{n} \beta_j \mu_{T,0}^{t_j,1}, \sum_{j=1}^{n} \beta_j K_T^{t_j,1}\right)\right)
\]
and \eqref{k_n_mu_n} follows from \eqref{equ:self inner product}.

\subsection{Proof of \eqref{q_jk}}\label{appendix:pf_q_jk} As in the proof of Proposition \ref{prop:Q_finiteD}, $q_{jk}$ is the real part of 
\begin{multline}\label{q_jk_1}
	\int_0^T \int_{\R} \Big\{ \frac{1}{2}  \partial_{s} \Fcal K_T^{t_j,1} (s,\xi) \overline{\partial_{s} \Fcal K_T^{t_k,1} (s,\xi)}
	\\
	+ \frac{1}{8} \big[- i \xi \Fcal \mu_{T,0}^{t_j,1} (\xi) + \xi^2  \Fcal K_T^{t_j,1} (s,\xi) \big] \overline{\big[- i \xi \Fcal \mu_{T,0}^{t_k,1} (\xi) + \xi^2  \Fcal K_T^{t_k,1} (s,\xi) \big]} \Big\} d\xi\,ds\\
	+ \frac{1}{4} \int_{\R} \Big\{\Fcal \mu_{T,0}^{t_j,1} (\xi) \overline{\Fcal \mu_{T,0}^{t_k,1} (\xi)} \\
	+ \big[\Fcal \mu_{T,0}^{t_j,1}  (\xi)  + i \xi  \Fcal K_{T}^{t_j,1} (T,\xi) \big] \overline{\big[\Fcal \mu_{T,0}^{t_k,1}  (\xi)  + i \xi  \Fcal K_{T}^{t_k,1} (T,\xi) \big]}\Big\} d \xi.
\end{multline}
Also recall from (P1) and (P2) after the proof of Lemma \ref{lem:Q_2} that
\begin{align*}
	\Fcal \mu^{t,\alpha}_{T,0} (\xi) &= (-i\xi) \frac{\alpha}{2\sqrt{t} \xi^2} [1-e^{-t\xi^2/2}],\\
	\Fcal K^{t,\alpha}_{T} (s, \xi) &= \frac{\alpha}{2\sqrt{t} \xi^2} [1-e^{-t\xi^2/2} + e^{-(t-s)\xi^2/2} -e^{-s\xi^2/2}], \quad 0 \leq s \leq t,\\
	\Fcal K^{t,\alpha}_{T} (s, \xi) &= \frac{\alpha}{2\sqrt{t} \xi^2} [1-e^{-t\xi^2/2} + e^{-(s-t)\xi^2/2} -e^{-s\xi^2/2}], \quad t \leq s \leq T.
\end{align*}
Without loss of generality, assume $t_j < t_k$.  Since for $0 \leq s \leq t$, 
\begin{align*}
	\partial_s \Fcal K^{t,\alpha}_{T} (s, \xi) &= \frac{\alpha}{4 \sqrt{t}} [ e^{-(t-s)\xi^2/2} + e^{-s\xi^2/2}],\\
	- i \xi \Fcal \mu^{t,\alpha}_{T,0} (\xi) + \xi^2 \Fcal K^{t,\alpha}_{T} (s, \xi)  &= \frac{\alpha}{2 \sqrt{t}} [ e^{-(t-s)\xi^2/2} -e^{-s\xi^2/2}],
\end{align*}
and for $t \leq s \leq T$, 
\begin{align*}
	\partial_s \Fcal K^{t,\alpha}_{T} (s, \xi) &= \frac{\alpha}{4 \sqrt{t}} [ -e^{-(s-t)\xi^2/2} + e^{-s\xi^2/2}],\\
	- i \xi \Fcal \mu^{t,\alpha}_{T,0} (\xi) + \xi^2 \Fcal K^{t,\alpha}_{T} (s, \xi)  &= \frac{\alpha}{2 \sqrt{t}} [ e^{-(s-t)\xi^2/2} -e^{-s\xi^2/2}],
\end{align*}
the first two lines in \eqref{q_jk_1} equal $A(0,t_j) + A(t_j,t_k) + A(t_k,T)$, where
\begin{multline*}
A(0,t_j)  = \frac{1}{32 \sqrt{t_jt_k}}\int_{0}^{t_j} \int_{\R} \Big\{ [ e^{-(t_j-s)\xi^2/2} + e^{-s\xi^2/2}][ e^{-(t_k-s)\xi^2/2} + e^{-s\xi^2/2}] \\
+ [ e^{-(t_j-s)\xi^2/2} -e^{-s\xi^2/2}][ e^{-(t_k-s)\xi^2/2} -e^{-s\xi^2/2}]\Big\}\,d\xi\,ds\\
=  \frac{1}{16 \sqrt{t_jt_k}}\int_{0}^{t_j} \int_{\R} \Big\{ e^{-(\tfrac{t_j+t_k}{2}-s)\xi^2} + e^{-s\xi^2}\Big\} \,d\xi\,ds = \frac{\sqrt{\pi}}{16 \sqrt{t_jt_k}}\int_{0}^{t_j} \Big\{ \frac{1}{\sqrt{\tfrac{t_j+t_k}{2}-s}} + \frac{1}{\sqrt{s}}\Big\}\,ds\\
=  \frac{\sqrt{\pi}}{8 \sqrt{t_jt_k}} \Big\{ \sqrt{\frac{t_j+t_k}{2}}  - \sqrt{\frac{t_k-t_j}{2}} + \sqrt{t_j}\Big\},
\end{multline*}
\begin{multline*}
A(t_j,t_k)  = \frac{1}{32 \sqrt{t_jt_k}}\int_{t_j}^{t_k} \int_{\R}   \Big\{ [ e^{-(t_k-s)\xi^2/2} + e^{-s\xi^2/2}][ -e^{-(s-t_j)\xi^2/2} + e^{-s\xi^2/2}] \\
+  [ e^{-(t_k-s)\xi^2/2} -e^{-s\xi^2/2}] [ e^{-(s-t_j)\xi^2/2} -e^{-s\xi^2/2}] \Big\} \,d\xi\,ds\\
= \frac{1}{16 \sqrt{t_jt_k}} \int_{t_j}^{t_k} \int_{\R} \Big\{ e^{-s \xi^2} - e^{-(s-\tfrac{t_j}{2})\xi^2}\Big\} \,d\xi\,ds = \frac{\sqrt{\pi}}{8 \sqrt{t_jt_k}} \Big\{  \sqrt{t_k} - \sqrt{t_j} - \sqrt{t_k - \frac{t_j}{2}} +\sqrt{\frac{t_j}{2}}\Big\},
\end{multline*}
and 
\begin{multline*}
	A (t_k,T) =  \frac{1}{32 \sqrt{t_jt_k}}\int_{t_k}^{T} \int_{\R}   \Big\{ [ -e^{-(s-t_k)\xi^2/2} + e^{-s\xi^2/2}] [ -e^{-(s-t_j)\xi^2/2} + e^{-s\xi^2/2}] \\
	+   [ e^{-(s-t_k)\xi^2/2} -e^{-s\xi^2/2}] [ e^{-(s-t_j)\xi^2/2} -e^{-s\xi^2/2}] \Big\} \,d\xi\,ds\\
	= \frac{1}{16 \sqrt{t_jt_k}} \int_{t_k}^{T} \int_{\R} \Big\{ e^{-(s-\tfrac{t_k+t_j}{2}) \xi^2} + e^{-s\xi^2} -e^{-(s-\tfrac{t_j}{2}) \xi^2} - e^{-(s-\tfrac{t_k}{2}) \xi^2}\Big\} 
	\,d\xi\,ds \\
	= \frac{\sqrt{\pi}}{8 \sqrt{t_jt_k}} \Big\{  \sqrt{T - \frac{t_k+t_j}{2}} -   \sqrt{\frac{t_k-t_j}{2}} + \sqrt{T} - \sqrt{t_k} - \sqrt{T - \frac{t_j}{2}} + \sqrt{t_k - \frac{t_j}{2}} - \sqrt{T - \frac{t_k}{2}}  + \sqrt{\frac{t_k}{2}} \Big\}.
\end{multline*}
Thus, the first two lines in \eqref{q_jk_1} equal 
\[\frac{\sqrt{\pi}}{8 \sqrt{t_jt_k}} \Big\{ \sqrt{\frac{t_j+t_k}{2}}  -2 \sqrt{\frac{t_k-t_j}{2}} + \sqrt{T}  + \sqrt{\frac{t_j}{2}} + \sqrt{\frac{t_k}{2}} + \sqrt{T - \frac{t_k+t_j}{2}}  - \sqrt{T - \frac{t_j}{2}}  - \sqrt{T - \frac{t_k}{2}}  \Big\}.\]
Similarly, since
\[	 \Fcal \mu^{t,\alpha}_{T,0} (\xi) + i \xi \Fcal K^{t,\alpha}_{T} (T, \xi)  = \frac{i \alpha}{2 \xi \sqrt{t}} [ e^{-(T-t)\xi^2/2} -e^{-T\xi^2/2}],\]
the last two lines in \eqref{q_jk_1} equal 
\[	\frac{1}{16 \sqrt{t_jt_k}} \int_{\R} \xi^{-2} \Big\{  [1-e^{-t_j\xi^2/2}][1-e^{-t_k\xi^2/2}]
+[ e^{-(T-t_j)\xi^2/2} -e^{-T\xi^2/2}] [ e^{-(T-t_k)\xi^2/2} -e^{-T\xi^2/2}]\Big\} \,d \xi .\]
Using the integeral formula
\[\int_{\R} \frac{1-e^{-ax^2}}{x^2} dx = 2 \sqrt{\pi a}, \quad a > 0,\]
the last two lines in \eqref{q_jk_1} equal 
\[	\frac{\sqrt{\pi}}{8 \sqrt{t_jt_k}} \Big\{ \sqrt{\frac{t_j}{2}} + \sqrt{\frac{t_k}{2}} - \sqrt{\frac{t_j+t_k}{2}} - \sqrt{T - \frac{t_j+t_k}{2}} - \sqrt{T} + \sqrt{T-\frac{t_k}{2}} + \sqrt{T-\frac{t_j}{2}}\Big\}.\]
Therefore,
\[q_{jk} = \frac{\sqrt{2 \pi}}{8 \sqrt{t_jt_k}} (\sqrt{t_k} + \sqrt{t_j} - \sqrt{t_k - t_j}),\]
as claimed.

\subsection{Proof of \eqref{uniform_see4}}\label{app_pf_uniform_see4} By Feynman-Kac formula (\cite[Lemma A.1.7.2]{klscaling}), we may bound the expression in \eqref{uniform_see4} by
\[\frac{kN^2}{a_N^2} \sup_{f:\nu_\rho-\text{density}}\Big\{ \int\frac{a_N K}{nN^2} \sum_{x=0}^{nN-1} \big(\eta(x) - \eta(x+1)\big) f(\eta)\nu_\rho (d \eta) - \<-\gen \sqrt{f},\sqrt{f}\>_\rho\Big\}.\]
Above, we call $f$ is a $\nu_\rho-$density if $f \geq 0$ and $\int f d \nu_\rho = 1$, and for two local functions $f,g: \Omega \rightarrow \R$,
\[\<f,g\>_\rho = \int f(\eta) g(\eta)\,\nu_\rho (d \eta).\]
Since $\nu_{\rho}$ is invariant for the SSEP,
\[ \<-\gen \sqrt{f},\sqrt{f}\>_\rho = \frac{1}{4} \sum_{x \in \Z} \int \big(\sqrt{f}(\eta^{x,x+1}) - \sqrt{f} (\eta)\big)^2 \nu_\rho (d \eta). \]
Making the change of variables $\eta \mapsto \eta^{x,x+1}$ and using the Cauchy-Schwarz inequality, for any $A > 0$, we bound the first term inside the above brace by
\begin{multline*}
	\frac{a_N K}{2nN^2} \sum_{x=0}^{nN-1} \int \big(\eta(x)-\eta(x+1)\big) \big( f (\eta) -f (\eta^{x,x+1}) \big) \nu_\rho (d \eta)
	\\
	\leq   \frac{a_N K}{4nN^2} \sum_{x=0}^{nN-1} \Big\{  A \int \big(\sqrt{f}(\eta^{x,x+1}) - \sqrt{f} (\eta)\big)^2 \nu_\rho (d \eta) + \frac{1}{A}   \int \big(\sqrt{f}(\eta^{x,x+1}) + \sqrt{f} (\eta)\big)^2 \nu_\rho (d \eta) \Big\}\\
	\leq  \frac{a_N KA}{ nN^2}  \<-\gen \sqrt{f},\sqrt{f}\>_\rho + \frac{a_NK}{NA}.
\end{multline*}
Taking $A = nN^2 / (a_N K)$, we finally bound the expression in \eqref{uniform_see4} by $kK^2/(Nn)$. Since $k \leq TN$, the proof is concluded.

\bibliographystyle{plain}
\bibliography{zhaoreference.bib}
\end{document}